\documentclass{amsart}

\usepackage{amsmath,amssymb,amsthm}
\usepackage{pinlabel}

\newtheorem{prop}{Proposition}[section]
\newtheorem{thm}[prop]{Theorem}
\newtheorem{lem}[prop]{Lemma}
\newtheorem{cor}[prop]{Corollary}

\theoremstyle{definition}

\newtheorem*{defn}{Definition}
\newtheorem*{ex}{Example}
\newtheorem*{exs}{Examples}
\newtheorem{rem}[prop]{Remark}
\newtheorem{rems}[prop]{Remarks}
\newtheorem*{nota}{Notation}
\newtheorem*{ack}{Acknowledgement}

\def\co{\colon\thinspace}

\newcommand{\BB}{\mathcal B}

\newcommand{\C}{\mathbb C}
\newcommand{\CC}{\mathcal C}

\newcommand{\rmd}{\mathrm d}
\newcommand{\D}{\mathbb D}

\newcommand{\rme}{\mathrm e}

\newcommand{\Hp}{\mathbb H}

\newcommand{\N}{\mathbb N}

\newcommand{\bfp}{\mathbf p}

\newcommand{\bfq}{\mathbf q}

\newcommand{\R}{\mathbb R}

\newcommand{\TT}{\mathcal T}

\newcommand{\UU}{\mathcal U}

\newcommand{\wtW}{\widetilde{W}}
\newcommand{\WW}{\mathcal W}

\newcommand{\Z}{\mathbb Z}

\newcommand{\lra}{\longrightarrow}
\newcommand{\ra}{\rightarrow}

\DeclareMathOperator{\diam}{\mathrm{diam}}

\DeclareMathOperator{\ev}{\mathrm{ev}}

\DeclareMathOperator{\Int}{\mathrm{Int}}

\DeclareMathOperator{\loc}{\mathrm{loc}}

\DeclareMathOperator{\st}{\mathrm{st}}
\DeclareMathOperator{\ot}{\mathrm{ot}}

\newcommand{\alst}{\alpha_{\st}}
\newcommand{\xist}{\xi_{\st}}
\newcommand{\lamst}{\lambda_{\st}}

\begin{document}

\author{Hansj\"org Geiges}
\author{Kai Zehmisch}
\address{Mathematisches Institut, Universit\"at zu K\"oln,
Weyertal 86--90, 50931 K\"oln, Germany}
\email{geiges@math.uni-koeln.de, kai.zehmisch@math.uni-koeln.de}


\title[How to recognise a $4$-ball]{How to recognise a $4$-ball when you
see one}

\date{}

\begin{abstract}
We apply the method of filling with holomorphic discs
to a $4$-dimensional symplectic cobordism with the standard contact
$3$-sphere as one convex boundary component. We establish the following
dichotomy: either the cobordism is diffeomorphic to a ball,
or there is a periodic Reeb orbit of quantifiably short
period in the concave boundary of the cobordism.
This allows us to give a unified treatment of
various results concerning Reeb dynamics on contact
$3$-manifolds, symplectic fillability, the topology of symplectic
cobordisms, symplectic non-squeezing, and the non-existence of exact
Lagrangian surfaces in standard symplectic $4$-space.
\end{abstract}

\subjclass[2010]{53D35; 37J45, 57R17}

\maketitle


\section{Introduction\label{intro}}
Ever since the work of Hofer~\cite{hofe93} on the Weinstein
conjecture for overtwisted contact $3$-manifolds, it has been a recurrent
theme in symplectic and contact topology that the non-compactness
of certain moduli spaces of holomorphic discs translates into
the existence of periodic Reeb orbits. For recent work in this
direction see for instance~\cite{alho09}.

The inspiration for Hofer's approach came from
Eliashberg's method of filling with holomorphic discs~\cite{elia90}.
In \cite{geze10} we gave a detailed discussion of that method
in a moduli-theoretic framework. As had been observed by Eliashberg,
a filling of the $4$-ball $D^4$ by holomorphic discs adapted to
a contactomorphism of the boundary $3$-sphere $S^3$ yields
a simple proof of Cerf's theorem that every diffeomorphism of
$S^3$ extends to a diffeomorphism of~$D^4$.

In the present paper we generalise the moduli-theoretic set-up
from \cite{geze10} to a disc-filling of a symplectic cobordism
that has the standard contact $3$-sphere $(S^3,\xist)$
as one convex boundary component. Our main result, which we shall
refer to as the `ball theorem', then says the following.
Either the corresponding moduli space of
holomorphic discs is compact, in which case the symplectic
cobordism has to be the $4$-ball, or there is non-compactness caused
by bubbling-off of holomorphic discs or breaking, in which
case there have to be periodic Reeb orbits in the concave boundary
of the symplectic cobordism. Energy estimates on the holomorphic
discs give rise to estimates on the periods of these Reeb orbits.

This ball theorem may be regarded as a generalisation of the
following fundamental results in $4$-dimen\-sio\-nal symplectic resp.\
$3$-dimensional contact topology:
\begin{itemize}
\item[-] Existence of periodic Reeb orbits on
star-shaped hypersurfaces in~$\R^4$, hypersurfaces of contact type, and
overtwisted contact $3$-manifolds
(Rabinowitz, Viterbo, Hofer).
\item[-] Topology of symplectic fillings of $(S^3,\xist)$
(Gromov, Eliashberg, McDuff).
\item[-] Tightness of weakly symplectically fillable
contact structures (Gromov, Eliashberg).
\item[-] Non-existence of exact Lagrangian surfaces in $\R^4$ (Gromov).
\end{itemize}
Indeed, all these results become straightforward consequences
of the ball theorem.

Our methods also yield some new results on the existence of contractible
periodic Reeb orbits. Moreover, our ball theorem allows us to
define a symplectic capacity via the periods of Reeb orbits
on contact type hypersurfaces. A simple computation of this
capacity for the $4$-ball and the cylinder over the $2$-ball
leads to a proof of
\begin{itemize}
\item[-] Symplectic non-squeezing (Gromov).
\end{itemize}

Conversely, this capacity can be used to give estimates
on the shortest Reeb period. We
recover some examples of Frauenfelder--Ginzburg--Schlenk
and provide additional information about the
periods of contractible orbits.

A precise description of the symplectic cobordisms we
are considering is given in Section~\ref{section:ball},
which also contains the statement of the ball theorem,
including a variant for symplectic cobordisms
with an exact symplectic form. Various corollaries of the ball
theorems, including the ones we just mentioned, will
be proved in Section~\ref{section:corollaries}.
The proof of the ball theorems is given in
Section~\ref{section:proof}, subject to
a compactness result for the relevant moduli space of holomorphic
discs. This compactness result is proved in Section~\ref{section:compact}
after a brief discussion of the Hofer energy in Section~\ref{section:hofer}.
It is worth pointing out that the larger part of our
compactness proof only involves classical bubbling-off
analysis as in~\cite{hofe93}; the new aspect here is that
we have to deal with bubbling at the boundary. For the interior
bubbling-off of spheres we rely on the more sophisticated
compactness results from~\cite{behwz03}. In a final section
we give a brief sketch how the filling with holomorphic discs
can be applied to weak symplectic fillings of $S^2\times S^1$
with its standard contact structure; as in the case of $S^3$ this
allows one to classify such fillings up to
diffeomorphism.

The set-up here is parallel to our previous paper~\cite{geze10};
it therefore seems opportune to list some minor
corrections to that paper in an appendix to the present one.
As in \cite{geze10} we write $\D\subset\C$ and $\Hp\subset\C$ for 
the {\em closed\/} unit disc and upper half-plane, respectively.
In \cite{geze12} we extend the results of the present paper to
higher dimensions.
\section{The ball theorems}
\label{section:ball}
We begin with a description of the specific symplectic cobordisms
that form the setting of our main theorems; see
Figure~\ref{figure:cobordism1}. For the basics of symplectic cobordisms
cf.~\cite[Chapter~5]{geig08}.

Let $(M_{\pm},\xi_{\pm}=\ker\alpha_{\pm})$
be two closed $3$-dimensional contact manifolds, oriented by the
volume forms $\alpha_{\pm}\wedge\rmd\alpha_{\pm}$.
The symplectic cobordisms $(W,\omega)$ we want to consider
are compact, connected symplectic $4$-manifolds, oriented by the
volume form~$\omega^2$, with the following properties:
\begin{itemize}
\item[(C1)] $(W,\omega)$ is minimal, i.e.\ does not contain
symplectically embedded $2$-spheres of self-intersection~$-1$
(so-called exceptional spheres).
\item[(C2)] The boundary of $W$ equals
\[ \partial W=\overline{M}_-\sqcup M_+\sqcup S^3\]
as oriented manifolds, where $\overline{M}_-$ denotes $M_-$ with the
reversed orientation. One or both of $M_{\pm}$ may be empty, and they
need not be connected.
\item[(C3)] The restriction of $\omega$ to (the tangent bundle of)
$M_-$ equals $\rmd\alpha_-$.
\item[(C4)] The restriction of $\omega$ to the $2$-plane field
$\xi_+=\ker\alpha_+$ on $M_+$ is positive.
\item[(C5)] A neighbourhood of $S^3\subset\partial W$ in $(W,\omega)$ looks
like a neighbourhood of $S^3=\partial D^4$ in $D^4$ with the standard
symplectic form $\omega_{\st}=\rmd x_1\wedge\rmd y_1+\rmd x_2\wedge\rmd y_2$.
\end{itemize}

Condition (C4), with the orientation condition~(C2),
says that $(M_+,\xi_+)$ is a weakly convex boundary
component of $(W,\omega)$; cf.~\cite[Chapter~5]{geig08} for the
various notions of convex resp.\ concave boundaries of symplectic
manifolds. The choice of contact form $\alpha_+$ defining
the given $\xi_+$ is irrelevant for our purposes.

Let
\[ \lamst:=\frac{1}{2}(x_1\,\rmd y_1-y_1\,\rmd x_1+
x_2\,\rmd y_2-y_2\,\rmd x_2) \]
be the standard primitive of the symplectic form~$\omega_{\st}$,
and set $\alst=\lamst|_{TS^3}$.
Condition (C5) says that $S^3$ with its standard contact
structure $\xist=\ker\alst$
is a strongly convex boundary of $(W,\omega)$, with a Liouville
vector field $Y$ for $\omega$ (i.e.\ $L_Y\omega =\omega$)
defined near $S^3\subset W$, pointing out of~$W$, and such
that $i_Y\omega$ restricts to the contact form $\alst$ on $S^3$.
This condition on the induced contact form, together with
condition~(C3), serves to normalise the contact form~$\alpha_-$,
which allows us to speak in quantitative terms about the
Reeb dynamics of~$\alpha_-$.

Finally, condition (C3) can be read as saying that $(M_-,\xi_-)$
is a strongly concave boundary of $(W,\omega)$. This is well known
and can best be seen with the help of relative de Rham cohomology.
With $U$ denoting a collar neighbourhood of $M_-$ in~$W$,
the relative de Rham cohomology group $H^2_{\mathrm{dR}}(U,M_-)$ is trivial.
From the definition of
relative de Rham cohomology, cf.~\cite[p.~78]{botu82}, one
finds a $1$-form $\beta$ on $U$ which restricts to $\alpha_-$
on (the tangent bundle of) $M_-$ and such that $\omega=\rmd\beta$ on~$U$.
The vector field $Y$ on $U$ defined by $i_Y\omega =\beta$
is then a Liouville vector field for $\omega$ that induces $\alpha_-$
on $M_-$ and points inward by the orientation condition~(C2).

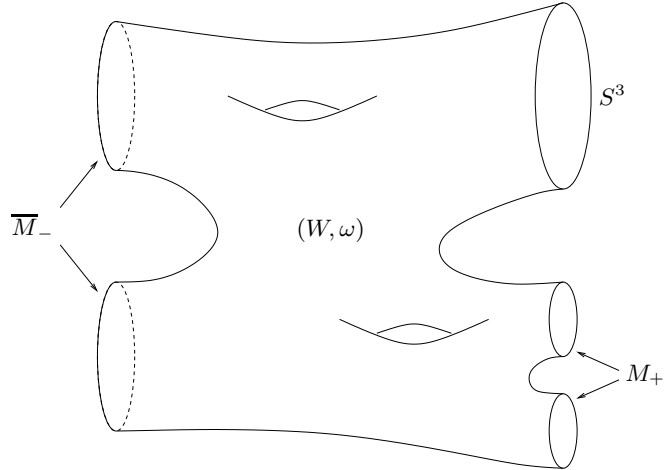
\begin{figure}[h]
\labellist
\small\hair 2pt
\pinlabel $\overline{M}_-$ [r] at 0 234
\pinlabel $M_+$ [l] at 542 91
\pinlabel $S^3$ [l] at 516 362
\pinlabel $(W,\omega)$ at 262 231
\endlabellist
\centering
\includegraphics[scale=0.4]{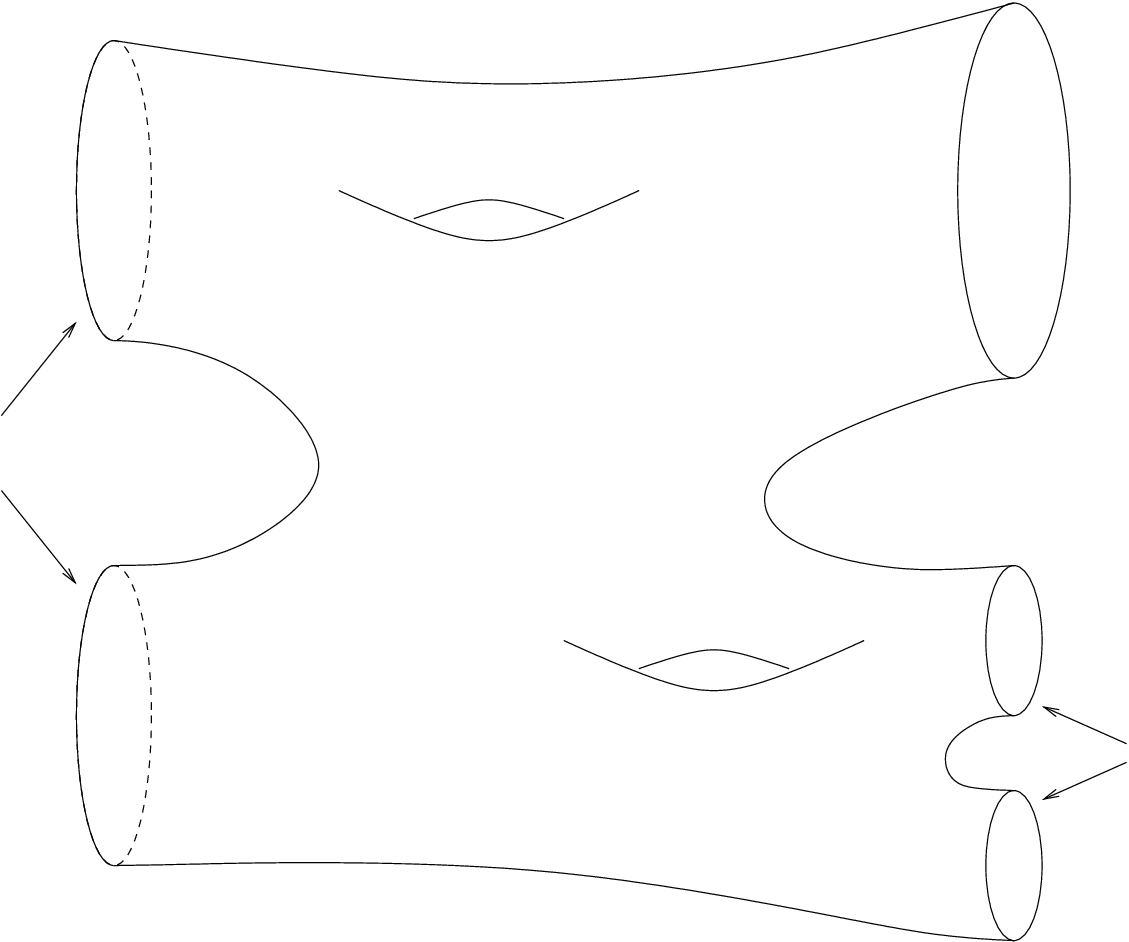}
  \caption{Might this be a $4$-ball?}
  \label{figure:cobordism1}
\end{figure}

Recall that the Reeb vector field $R=R_{\alpha}$ of a contact form $\alpha$
is defined by the equations $i_R\rmd\alpha=0$ and $\alpha (R)=1$.

\begin{nota}
We write $\inf (\alpha )$ for the infimum of all positive periods
of closed orbits of the Reeb vector field~$R_{\alpha}$.
With $\inf_0(\alpha)$ we denote the infimum of all positive periods
of {\em contractible} closed Reeb orbits.
\end{nota}

\begin{rem}
An argument as in \cite[p.~109]{hoze94} shows that both
infima are minima, and in particular positive, unless the
relevant set of Reeb orbits is empty and the infimum equal to~$\infty$.
\end{rem}

We can now state our two main theorems, whose essence is the following:
unless $W$ is a $4$-ball (and in particular $M_{\pm}$ are empty),
there must be a short Reeb orbit on~$(M_-,\alpha_-)$.

\begin{thm}[The ball theorem]
\label{thm:ball}
Let $(W,\omega)$ be a symplectic cobordism satisfying conditions
{\rm (C1)} to {\rm (C5)}. Then either $\inf(\alpha_-)\leq\pi$ or
$W$ is diffeomorphic to a $4$-ball.
\end{thm}

\begin{rem}
In the case that $W$ is a $4$-ball, a theorem of
Gromov~\cite[p.~311]{grom85} implies that $(W,\omega)$ is
actually {\em symplectomorphic\/} to $(D^4,\omega_{\st})$.
\end{rem}

The following version of the ball theorem sharpens the dichotomy
under the additional requirement that the symplectic form have
a suitable primitive.

\begin{thm}[The exact ball theorem]
\label{thm:ball-exact}
Let $(W,\omega=\rmd\lambda)$ be a symplectic cobordism satisfying conditions
{\rm (C1)} to {\rm (C5)}, and with $\lambda|_{TM_-}=\alpha_-$.
Then either $\inf_0(\alpha_-)\leq \pi$ or $W$ is
diffeomorphic to a $4$-ball.
\end{thm}

By the theorem of Stokes there are no symplectically embedded
$2$-spheres in an exact symplectic manifold, so condition
(C1) is automatic in the exact case.

The dichotomy in the ball theorems can be exploited either way:
from the non-existence of short periodic Reeb orbits (e.g.\ when
$M_-$ is empty) one can deduce topological information
about symplectic cobordisms; conversely, topological information
about a symplectic cobordism can be used to detect a short closed
Reeb orbit. A number of such results will be derived in
Section~\ref{section:corollaries}. In several of these corollaries
the following concept and the construction that we shall describe
presently play an important role.

\begin{defn}
Let $(M_{\pm},\xi_{\pm})$ be closed $3$-dimensional contact manifolds.
A compact symplectic $4$-manifold $(W,\omega)$ with oriented boundary
$\partial W=\overline{M}_-\sqcup M_+$ is called a {\bf Liouville
cobordism} from $M_-$ to $M_+$ if the symplectic form $\omega$ is exact and
its primitive can be chosen as a contact form for~$\xi_{\pm}$, i.e.\
$\omega=\rmd\lambda$ with $\ker(\lambda|_{TM_{\pm}})=\xi_{\pm}$.
\end{defn}

The following example (also observed by Wendl~\cite{wend}) shows that
a symplectic cobordism with an exact symplectic form
is not, in general, a Liouville cobordism.

\begin{ex}
By the Weinstein tubular neighbourhood theorem~\cite{wein71},
the complement of a tubular neighbourhood of a Lagrangian $2$-torus
in $(D^4,\rmd\lamst)$ is a strong symplectic cobordism from
$(T^3,\ker(\cos\theta\,\rmd x-\sin\theta\,\rmd y))$ to $(S^3,\xist)$.
By the exact ball theorem, however, there can be no
Liouville cobordism, since the periodic Reeb orbits of the contact form
$\cos\theta\,\rmd x-\sin\theta\,\rmd y$ are all non-contractible.
\end{ex}

\begin{rem}
\label{rem:exact}
Given a symplectic cobordism $(W,\omega)$ with a weakly convex
boundary component $(M_+,\xi_+)$ and a symplectic form that
is exact near~$M_+$, a construction of
Eliashberg~\cite{elia91}, cf.~\cite{geig06}, allows one to modify
the primitive $\lambda$ of $\omega$ (defined in a neighbourhood of~$M_+$)
in such a way that $\ker(\lambda|_{TM_+})=\xi_+$.
As the preceding example shows, such a modification is not
possible, in general, at a concave end.
Moreover, one can then arrange $\lambda$ to equal
a given contact form $\alpha_+$ on $M_+$
up to a constant scale factor by first taking the symplectic completion
of $(W,\omega)$ along~$M_+$, i.e.\ adding a cylindrical end of the
form $([0,\infty)\times M_+,\rmd (\rme^s\lambda|_{TM_+}))$,
and then replacing $M_+\equiv\{ 0\}\times M_+$
by a suitable graph in this cylindrical end.
\end{rem}

Strong symplectic and Liouville cobordisms can be glued
together (along a convex and a contactomorphic concave end) by using
the Liouville vector field to define collar neighbourhoods of the
boundaries. This may require the rescaling of one of the
symplectic forms by a constant and the insertion of
a cylindrical tube, see~\cite[Proposition~5.2.5]{geig08}.
\section{Corollaries of the ball theorems}
\label{section:corollaries}
\subsection{Topology of symplectic cobordisms}
Our first corollary was originally proved by
McDuff~\cite[Theorem~1.2]{mcdu91}.

\begin{cor}[McDuff]
\label{cor:McDuff}
If $(W,\omega)$ is a compact symplectic $4$-manifold
with weakly convex boundary components only, and
one of the boundary components is $(S^3,\xist)$, then
the boundary is connected.
\end{cor}

\begin{proof}
After blowing down exceptional spheres in $(W,\omega)$,
cf.~\cite[Chapter~7]{mcsa98}, we may assume that
$(W,\omega)$ is minimal. Since $H^2_{\mathrm{dR}}(S^3)=0$,
the symplectic form $\omega$
is exact in a neighbourhood of the boundary component~$S^3$.
Hence, by Remark~\ref{rem:exact},
$\omega$ can be modified in that neighbourhood so
that condition (C5) from Section~\ref{section:ball} is satisfied (up to
a constant scale factor). We are then in the situation of
Theorem~\ref{thm:ball} with $M_-=\emptyset$. This implies
$\inf(\alpha_-)=\infty$, and so the theorem tells us that
$W\cong D^4$, which means $M_+=\emptyset$.
\end{proof}

This proof also yields the following
variant of a result due to Gromov~\cite[p.~311]{grom85},
Eliashberg~\cite[Theorem~5.1]{elia90} and McDuff~\cite[Theorem~1.7]{mcdu90}.

\begin{cor}[Gromov, Eliashberg, McDuff]
Any minimal weak symplectic filling of $(S^3,\xist)$ is
diffeomorphic to the $4$-ball.\qed
\end{cor}

Before we turn to the next corollary, we recall a statement
about symplectic cobordisms that will be used in the proof of that
and other corollaries. This statement is originally due to
Etnyre--Honda~\cite{etho02}; here we give a proof based on
the surgery presentation theorem~\cite{dige04} for contact $3$-manifolds.

\begin{thm}[Etnyre--Honda]
\label{thm:etho}
Let $(M_{\ot},\xi_{\ot})$ be an overtwisted contact $3$-mani\-fold,
and $(M,\xi)$ any contact $3$-manifold, where both manifolds are assumed
to be closed and connected. Then there is a Liouville cobordism from
$(M_{\ot},\xi_{\ot})$ to $(M,\xi)$.
\end{thm}

\begin{proof}
It suffices to show that $(M,\xi)$ can be obtained from
$(M_{\ot},\xi_{\ot})$ by a sequence of Legendrian surgeries
(or contact $(-1)$-surgeries in the sense of~\cite{dige04}),
since such surgeries translate into a Liouville cobordism from
the original to the surgered manifold.

By the classical surgery presentation theorem for $3$-manifolds
due to Lickorish and Wallace, there is a sequence of integer
surgeries that gets us from $M_{\ot}$ to~$M$.
By a result of Eliashberg~\cite{elia90a}, cf.\ \cite[Chapter~6.3]{geig08},
in an overtwisted contact manifold we can choose a Legendrian
realisation of the surgery link in such a way that the
desired integer surgeries correspond to contact $(-1)$-surgeries
along the components of the Legendrian link.
This yields a Liouville cobordism from $(M_{\ot},\xi_{\ot})$
to $(M,\xi')$, where $\xi'$ is some contact structure on~$M$.
By adding homotopically trivial Lutz twists on
$(M_{\ot},\xi_{\ot})$ away from the surgery link (this does not
change $\xi_{\ot}$ by Eliashberg's classification~\cite{elia89} of overtwisted
contact structures), we can ensure that $\xi'$ is likewise
overtwisted. This means that $\xi'$ can be obtained from $\xi$
by performing topologically trivial Lutz twists,
which can be realised as contact $(+1)$-surgeries~\cite{dige04}.
Conversely, by the cancellation lemma from~\cite[Section~3]{dige04},
cf.\ \cite[Proposition~6.4.5]{geig08},
$(M,\xi)$ is obtained from $(M,\xi')$ by contact $(-1)$-surgeries.
\end{proof}

The part of the next corollary concerning overtwisted
contact structures is contained in the work of Hofer~\cite{hofe93}.

\begin{cor}[Hofer]
\label{cor:Hofer1}
If $\xi_-$ is a contact structure on a closed $3$-manifold $M_-$
that can be defined by a contact form $\alpha_-$ without contractible
periodic Reeb orbits, then there is no Liouville
cobordism from $(M_-,\xi_-)$ to a not necessarily connected
contact $3$-manifold with at least one overtwisted component.
In particular, $\xi_-=\ker\alpha_-$ is tight, i.e.\ not overtwisted.
\end{cor}

\begin{proof}
Suppose $(M_-,\xi_-=\ker\alpha_-)$ admits a Liouville cobordism to a 
contact manifold having an overtwisted component $(M_{\ot},\xi_{\ot})$.
According to Theorem~\ref{thm:etho}, there is
a Liouville cobordism from $(M_{\ot},\xi_{\ot})$ to $(S^3,\xist)$.
By gluing this Liouville cobordism to the given one
(and modifying the boundaries inside the symplectic completion as
in Remark~\ref{rem:exact}), we obtain a cobordism as in the
exact ball theorem, up to constant scale of the contact forms on
the boundary. That theorem then guarantees the existence of a
contractible periodic Reeb orbit for~$\alpha_-$.
\end{proof}

\begin{rem}
Observe that the essence of Corollary~\ref{cor:Hofer1} is that
any contact form defining an overtwisted contact structure
on a closed $3$-manifold has a contractible periodic Reeb orbit.
\end{rem}
\subsection{Tightness and fillability}
The following corollary belongs to Gromov~\cite[$2.4.\mathrm{D}_2'$]{grom85}
and Eliashberg~\cite[Theorem~3.2.1]{elia88}.

\begin{cor}[Gromov, Eliashberg]
Let $\xi$ be a contact structure on a closed $3$-manifold~$M$.
If $(M,\xi)$ is weakly symplectically fillable, then $\xi$
is tight.
\end{cor}

\begin{proof}
We argue by contradiction. Assume that we have a weak symplectic filling
$(W_1,\omega_1)$ of an overtwisted contact $3$-manifold $(M,\xi)$.
Let $(W_2,\omega_2)$ be a compact symplectic $4$-manifold with
disconnected boundary of contact type (i.e.\ a strong
filling), as constructed
in \cite{mcdu91} or~\cite{geig95}. Take the boundary connected
sum of $(W_1,\omega_1)$ with $(W_2,\omega_2)$ along $(M,\xi)$ and
one of the boundary components of $(W_2,\omega_2)$. The result will
be a weak symplectic filling of a disconnected contact manifold,
one of whose components is an overtwisted contact manifold
$(M_{\ot},\xi_{\ot})$.

The Liouville cobordism from $(M_{\ot},\xi_{\ot})$ to
$(S^3,\xist)$ from Theorem~\ref{thm:etho} is made up
of symplectic $2$-handles; by \cite[Lemma~6.5.2]{geig08} such a cobordism
can also be attached to a weak filling.
The resulting symplectic manifold contradicts Corollary~\ref{cor:McDuff}.
\end{proof}

\subsection{Lagrangian surfaces in~$\R^4$}
Let $i\co\Sigma\hookrightarrow (W,\omega=\rmd\lambda)$ be a Lagrangian
embedding into an exact symplectic manifold, i.e.\ $i^*\omega=0$, which means
that $i^*\lambda$ is closed, and $\dim\Sigma=(\dim W)/2$.
Such an embedding is called {\bf exact} if $i^*\lambda$ is an
exact $1$-form.

Gromov~\cite[Corollary~$2.3.\mathrm{B}_2$]{grom85} has shown that there
are no closed exact Lagrangian submanifolds in $\R^{2n}$
with its standard symplectic structure. The exact ball theorem allows us to
prove this result in~$(\R^4,\rmd\lamst)$.

\begin{cor}[Gromov]
There are no closed exact Lagrangian surfaces in standard
symplectic $4$-space.
\end{cor}

\begin{proof}
A bundle-theoretic argument shows that a necessary condition
for a closed surface $\Sigma$ to admit a Lagrangian embedding in $\R^4$
is that $\Sigma$ be a torus or a non-orientable surface
of Euler characteristic divisible by~$4$, cf.~\cite[Section~3.2]{alp94}.
Moreover, all these surfaces, except the Klein bottle,
actually admit a Lagrangian embedding~\cite{give86,nemi09,shev09}.

We prove the corollary by contradiction. Our argument applies to
all surfaces of genus at least~$1$, which by the foregoing remark
covers all potential cases. Thus, suppose that $i\co\Sigma\hookrightarrow
(\R^4,\rmd\lamst)$ is an exact Lagrangian embedding of such a
surface~$\Sigma$. Write $i^*\lamst=\rmd f$ for some smooth function
$f$ on~$\Sigma$. By Weinstein's neighbourhood theorem~\cite{wein71},
a small tubular neighbourhood $U$ of $i(\Sigma)$ in $(\R^4,\omega_{\st})$
may be identified symplectomorphically with a neighbourhood of the zero
section in the cotangent bundle $T^*\Sigma$ with its standard
symplectic form $\rmd\lambda_0$, where $\lambda_0$ is the Liouville
$1$-form $\bfp\,\rmd\bfq$ (in local coordinates $\bfq$
on $\Sigma$ and their dual coordinates~$\bfp$).

Under the identification provided by Weinstein's theorem, we regard
$\lambda_0$ as a $1$-form defined on $U\subset\R^4$. Equip $\Sigma$
with the Riemannian metric of constant curvature $K\leq 0$
(here we use the genus restriction); this induces
a bundle metric on $T^*\Sigma$. Choose
a smaller tubular neighbourhood~$U_0$ that corresponds to
a disc neighbourhood $\{ \|\mathbf{p}\|<\varepsilon\}$
in the Weinstein model and whose closure is contained in~$U$.
Then the boundary $\partial U_0$
is transverse to the radial Liouville vector field
$\mathbf{p}\,\partial_{\mathbf{p}}$ for $\rmd\lambda_0$ in the Weinstein
model. In particular, the restriction of $\lambda_0$ to $T\partial U_0$
is a contact form.

We have $\rmd (\lamst-\lambda_0)=0$ on~$U$.
This implies that the $1$-form $\lamst-\lambda_0$ represents
a de Rham cohomology class $[\lamst-\lambda_0]\in H^1_{\mathrm{dR}}(U)
\cong  H^1_{\mathrm{dR}}(\Sigma)$. Moreover, we have $i^*\lamst=\rmd f$
and $i^*\lambda_0=0$, hence $i^*[\lamst-\lambda_0]=0\in
H^1_{\mathrm{dR}}(\Sigma)$. It follows that $[\lamst-\lambda_0]=0\in
H^1_{\mathrm{dR}}(U)$, so there is a smooth function $g$ on $U$
such that $\lamst-\lambda_0=\rmd g$. Let $\tilde{g}$ be a smooth
interpolation between $0$ on $U_0$ and $g$ near the boundary of~$U$.
Then $\lambda_0+\rmd\tilde{g}$ defines a primitive of $\omega_{\st}$
that coincides with $\lambda_0$ on~$U_0$, and with $\lamst$
near the boundary of~$U$, and so extends to a global primitive $\lambda$
of~$\omega_{\st}$.

Now let $S^3_R\subset\R^4$ be the sphere of radius~$R$ (centred at~$0$),
where $R$ is chosen so large that $U$ is contained in the
interior of~$S^3_R$. Then the complement of $U_0$ in the $4$-ball $D^4_R$
of radius $R$ with the symplectic form $\rmd\lambda$ constitutes a
Liouville cobordism between $(\partial U_0,\lambda_0|_{T\partial U_0})$
and $(S^3,R^2\alst)$. So the desired contradiction will follow
from the exact ball theorem, provided we can show
there are no contractible periodic Reeb orbits on~$\partial U_0$.

The Reeb flow on $\partial U_0$ corresponds to the geodesic flow
on the unit tangent bundle of~$\Sigma$, cf.~\cite[Theorem~1.5.2]{geig08}.
Hence, a contractible periodic Reeb orbit would correspond to
a contractible closed geodesic on~$\Sigma$, which in turn would
lift to a closed geodesic on the universal cover of~$\Sigma$. This is
clearly impossible when that cover is the Euclidean or hyperbolic plane.
\end{proof}
\subsection{Reeb dynamics}
In \cite{rabi79} Rabinowitz showed the existence of periodic
solutions of the Hamiltonian equation on $\R^{2n}$ on any
star-shaped level surface of any given Hamiltonian function.
This led Weinstein~\cite{wein79} to conjecture (in modern parlance)
the existence of closed Reeb orbits on arbitrary contact type
hypersurfaces in symplectic manifolds.

In dimension~3 this conjecture has been resolved positively
by Taubes~\cite{taub07}, using Seiberg--Witten--Floer theory.
Our ball theorems allow us to retrace most of the earlier
results in the history of the Weinstein conjecture, and
they yield new existence statements about {\em contractible}
periodic orbits.

\begin{cor}[Rabinowitz]
\label{cor:Rabinowitz}
Let $S\subset\R^4$ be a smooth hypersurface bounding a
domain star-shaped with respect to $0\in\R^4$. Then the contact form
$\lamst|_{TS}$ has a closed and obviously contractible Reeb orbit.
\end{cor}

\begin{proof}
Let $S^3_R\subset\R^4$ be the sphere of radius~$R$ (centred at~$0$),
where $R$ is chosen so large that $S$ is contained in the
interior of~$S^3_R$. Then the region $W$ between $S$ and $S^3_R$
with the symplectic form $\rmd\lamst$ constitutes a symplectic cobordism
between $(S,\lamst|_{TS})$ and $(S^3,R^2\alst)$ as in the ball theorems.
\end{proof}

The contact structure $\ker(\lamst|_{TS})$ in the theorem of Rabinowitz is
always diffeomorphic to the standard tight contact structure $\xist$
on~$S^3$. Hofer~\cite{hofe93} was the first to prove the Weinstein conjecture
for arbitrary contact forms on~$S^3$.

\begin{cor}[Hofer]
\label{cor:Hofer2}
The Reeb vector field of any contact form on $S^3$ has
a periodic Reeb orbit.
\end{cor}

\begin{proof}
For contact forms defining an overtwisted contact structure,
this is contained in Corollary~\ref{cor:Hofer1}.

For tight contact structures one has the following argument
from~\cite{hofe93}.  As shown by Eliashberg~\cite{elia92}, there is a unique
positive, (co-)oriented tight contact structure on~$S^3$ up to isotopy.
Thus, if $\alpha$ is a contact form defining any tight contact structure,
there is a diffeomorphism $\varphi$ of $S^3$ such that
$\varphi^*\alpha=f\alst$ for some smooth function $f\co S^3\ra\R^+$.

Now consider the star-shaped hypersurface
\[ S:=\{ \sqrt{f(p)}\, p\co p\in S^3\}.\]
According to Corollary~\ref{cor:Rabinowitz}, the contact form
$\lamst|_{TS}$ has a periodic Reeb orbit.
Under the map $S^3\ra S$, $p\mapsto\sqrt{f(p)}\, p$,
the $1$-form $\lamst|_{TS}$ pulls back to $f\alst$. So $f\alst$
and hence $\alpha$ likewise have periodic Reeb orbits.
\end{proof}

In a different direction, the result of Rabinowitz was extended
by Viterbo~\cite{vite87}.

\begin{defn}
A hypersurface $M$ in a symplectic manifold $(W,\omega)$
is said to be of {\bf contact type} (or locally $\omega$-convex)
if there is a Liouville vector field for $\omega$ defined near and
transverse to~$M$. The hypersurface is said to be of
{\bf restricted contact type} (or globally $\omega$-convex)
if the Liouville vector field is defined on all of~$W$.
\end{defn}

Viterbo proved the Weinstein conjecture for compact contact type
hypersurfaces in standard symplectic~$\R^{2n}$. Our ball theorem covers
this result in dimension~$4$.

\begin{cor}[Viterbo]
Let $M_-\subset (\R^4,\omega_{\st})$ be a smooth compact hypersurface
and $Y$ a Liouville vector field for $\omega_{\st}$ defined near
and transverse to~$M_-$. Then the contact form
$\alpha_-:=(i_Y\omega_{\st})|_{TM_-}$ has a periodic Reeb orbit.
\end{cor}

\begin{proof}
Without loss of generality we take $M_-$ to be connected.
Then this hypersurface separates $\R^4$ into a bounded and an
unbounded part. Choose a large sphere $S^3_R$ containing
$M_-$ in the interior, and write $W$ for the part between $M_-$ and~$S^3_R$.

The Liouville vector field $Y$ near $M_-$ points into~$W$,
otherwise Corollary~\ref{cor:McDuff} would be violated.
So the ball theorem applies.
\end{proof}

\begin{rem}
A neighbourhood of $M_-\subset (\R^4,\omega_{\st})$ looks like a
neighbourhood of $\{0\}\times M_-$ in the symplectisation
$(\R\times M_-,\rmd (\rme^s\alpha_-))$.
So we can form the symplectic manifold $(-\infty,0]\times M_-
\cup_{M_-}W$, with $W$ as in the preceding proof and
symplectic form $\rmd (\rme^s\alpha_-)$ on $(-\infty,0]\times M_-$.
Any contact form defining the contact structure
$\ker\alpha_-$ can be realised, up to a constant scale,
on a graph in this half-symplectisation. So the theorem holds
for any such contact form.
\end{rem}

For hypersurfaces of restricted contact type we get a stronger result.

\begin{cor}
Let $M_-\subset (\R^4,\omega_{\st})$ be a smooth compact hypersurface
and $Y$ a Liouville vector field for $\omega_{\st}$ defined on all
of $\R^4$ and transverse to~$M_-$. Then the contact form
$\alpha_-:=(i_Y\omega_{\st})|_{TM_-}$ has a contractible periodic Reeb orbit.
\end{cor}

\begin{proof}
Choose $S^3_R$ as in the preceding proof. The symplectic form $\omega_{\st}$
has the two global primitives $i_Y\omega_{\st}$ and $\lamst$.
Since $H^1_{\mathrm{dR}}(S^3_R)=0$, the difference $i_Y\omega_{\st}-\lamst$
is exact in a neighbourhood of $S^3_R$. So we can easily
construct a primitive of $\omega_{\st}$ that coincides with
$i_Y\omega_{\st}$ near $M_-$ and with $\lamst$ near~$S^3_R$.
Then the result follows from the exact ball theorem.
\end{proof}

Implicit in that argument is the simple observation that
a hypersurface of contact type with $H^1_{\mathrm{dR}}=0$ is
automatically of restricted contact type. But there are examples of
hypersurfaces of restricted contact type with $H^1_{\mathrm{dR}}\neq 0$,
for instance the connected sum of copies of $S^2\times S^1$;
this example can be constructed with the help
of \cite[Th\'eor\`eme~1]{laud97}.
So the existence of a {\em contractible} periodic orbit is
not just a consequence of topology.
On the other hand, the connected sum of copies of $S^2\times S^1$ has
non-trivial second homotopy group. So here the existence of a
contractible periodic Reeb orbit also follows from
Hofer's work~\cite[Theorem~9]{hofe93}.

The next proposition gives a further surgical construction
of contact manifolds having contractible periodic Reeb orbits.

\begin{prop}
\label{prop:Reeb-surgery}
Let $(M,\xi)$ be a closed contact $3$-manifold that is obtained
from $(S^3,\xist)$ by contact $(+1)$-surgery along a Legendrian
link. Then every contact form defining $\xi$ has a contractible
periodic Reeb orbit.
\end{prop}

\begin{proof}
By the cancellation lemma \cite[Proposition~6.4.5]{geig08} the
assumption of the proposition is equivalent to saying that
$(S^3,\xist)$ can be obtained from $(M,\xi)$ by contact $(-1)$-surgeries.
This means that there is a Liouville cobordism from $(M,\xi)$
to $(S^3,\xist)$. As described at the end of Section~\ref{section:ball},
this allows us to build a cobordism as in the exact ball theorem
for any choice of contact form defining~$\xi$.
\end{proof}

Here are three examples to which this proposition applies.

\begin{exs}
(1) Contact $(+1)$-surgery on $(S^3,\xist)$ along a standard Legendrian
unknot yields $S^2\times S^1$ with its
standard contact structure (as described in Section~\ref{section:S2S1}
below), see~\cite[Lemma~4.3]{dgs04}. For this example the existence of
a contractible periodic Reeb orbit also follows directly from~\cite{hofe93},
where Hofer proved the Weinstein conjecture for $3$-manifolds with
non-trivial second homotopy group.

(2) Contact $(+1)$-surgery on the Legendrian realisation of
the right-handed trefoil as in Figure~\ref{figure:trefoil}
(showing the front projection of that Legendrian knot)
produces a tight contact structure on the Brieskorn manifold
$\Sigma(2,3,4)$ with the opposite of its natural
orientation, see~\cite[p.~206]{ozst04}.
The universal cover is~$S^3$,
see~\cite{miln75}, so $\Sigma(2,3,4)$  has trivial second homotopy group.

\begin{figure}[h]
\labellist
\small\hair 2pt
\pinlabel $+1$ [tl] at 357 117
\endlabellist
\centering
\includegraphics[scale=0.5]{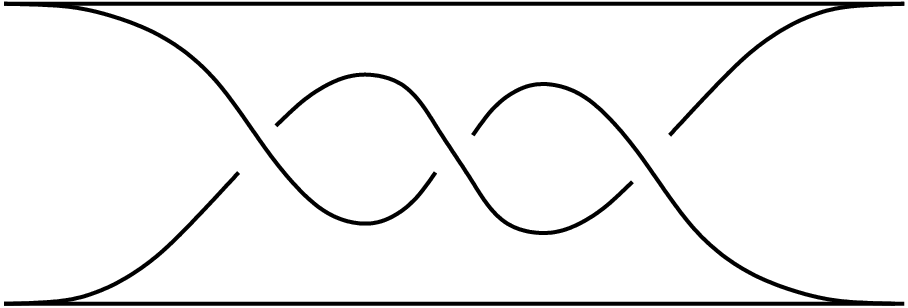}
  \caption{A tight contact structure on $\overline{\Sigma} (2,3,4)$.}
  \label{figure:trefoil}
\end{figure}

(3) Figure~12.4 of \cite{ozst04} gives an example of a tight
contact structure on the circle bundle of Euler number~$2$
over the torus, obtained by performing contact $(+1)$-surgeries
on a Legendrian link in $(S^3,\xist)$.
The second homotopy group of this manifold is trivial,
since its universal cover is~$\R^3$.
\end{exs}

\begin{rem}
Example~(2) is finitely covered by the $3$-sphere; example~(3)
is virtually overtwisted, i.e.\ finitely covered by an overtwisted
contact manifold. On these covers, a contractible Reeb orbit
is guaranteed by Corollaries~\ref{cor:Hofer2} and~\ref{cor:Hofer1},
respectively. This implies the existence of a contractible periodic Reeb
orbit downstairs. This contractible orbit may be a multiply
covered one. In example~(3) even the singly covered
orbit will be contractible, since the fundamental group
is torsion-free.
\end{rem}
\subsection{Capacities and non-squeezing}
Let $(V,\omega )$ be any $4$-dimensional symplectic manifold.
The manifold $V$ may be non-compact and disconnected.
For simplicity we assume that $V$ does not have boundary;
otherwise replace $V$ by $\Int V$ in the following definitions.
We define the following symplectic invariant of $(V,\omega)$:
\[ c(V,\omega):=\sup_{(M,\alpha)}\{\inf (\alpha)|\,
\text{$\exists$ contact type embedding
$(M,\alpha)\hookrightarrow (V,\omega)$}\} .\]
Here the supremum is taken over all closed, but not necessarily
connected contact $3$-manifolds $(M,\alpha)$. By a contact type
embedding $j\co (M,\alpha)\hookrightarrow (V,\omega)$ we mean that there
is a Liouville vector field $Y$ for $\omega$ defined near $j(M)$
such that $j^*(i_Y\omega)=\alpha$.

When $\omega=\rmd\lambda$ is exact, we can define the following
invariant:
\[ c_0(V,\lambda):=\sup_{(M,\alpha)}\{{\inf}_0 (\alpha)|\,
\text{$\exists$ embedding
$j\co (M,\alpha)\hookrightarrow (V,\rmd\lambda)$
with $j^*\lambda=\alpha$}\} .\]
In other words, here the supremum is taken over $(M,\alpha)$
admitting a restricted contact type embedding $j$ into $(V,\rmd\lambda)$,
where the global primitive $\lambda$ is fixed {\em a priori}.

In $\R^4$ with the standard symplectic form $\omega_{\st}=\rmd\lamst$ let
$B^4_r$ be the open $4$-ball of radius $r$ and $Z_r=B^2_r
\times\R^2$ the cylinder over the open $2$-ball of radius~$r$.
For $r=1$ we simply write $B^4$ and $Z$, respectively.

\begin{prop}
\label{prop:capacity}
The invariants $c(V,\omega)$ and $c_0(V,\lambda)$ are
symplectic capacities, i.e.\ they satisfy the following axioms:
\begin{description}
\item[Monotonicity] If there
exists a symplectic embedding $(V,\omega)\hookrightarrow (V',\omega')$,
then $c(V,\omega)\leq c(V',\omega')$;
similarly $c_0(V,\lambda )\leq c_0(V',\lambda')$
if there exists a symplectic embedding
$(V,\rmd\lambda )\hookrightarrow (V',\rmd\lambda')$ pulling back $\lambda'$
to~$\lambda$.
\item[Conformality] For any $a\in\R^+$ we have
$c(V,a\omega)=a\, c(V,\omega)$ and $c_0(V,a\lambda)=a\, c_0(V,\lambda)$.
\item[Normalisation] $c(B^4)=c_0(B^4)=c(Z)=c_0(Z)=\pi$.
\end{description}
\end{prop}

\begin{proof}
Monotonicity and conformality are obvious from the definition.
Write $S^3_r$ for the $3$-sphere of radius $r$, and denote
$\lamst|_{TS^3_r}$ by~$\alpha_r$. The Reeb vector field $R_r$
of $\alpha_r$ is given by
\[ R_r=\frac{2}{r^2}(x_1\,\partial_{y_1}-y_1\,\partial_{x_1}+
x_2\,\partial_{y_2}-y_2\,\partial_{x_2});\]
this has length $2/r$. All the orbits of $R_r$ are closed of
length $2\pi r$, so the period is $\pi r^2$. Since
$(S^3_r,\alpha_r)$ for $r<1$ has a (strict) contact type embedding
into the four manifolds we are considering
(with the symplectic form $\omega_{\st}$ and the global primitive~$\lamst$),
all four capacities are bounded from below by~$\pi$.

Suppose we have a (strict) contact type embedding
$j\co (M,\alpha )\hookrightarrow B^4$.
Then we get a cobordism from $j(M)$ to $S^3$ as in the
(exact) ball theorem, and these theorems tell us
that $\inf(\alpha),\inf_0(\alpha)\leq\pi$, since the cobordism is not
a ball. This concludes the proof of $c(B^4)=c_0(B^4)=\pi$.

If we have a (strict) contact type embedding
$j\co (M,\alpha )\hookrightarrow Z$,
the image $j(M)$ is contained inside an ellipsoid
\[ E(1,b)=\Bigl\{ x_1^2+y_1^2+\frac{x_2^2+y_2^2}{b^2}\leq 1\Bigr\} \]
for $b>0$ sufficiently large. The boundary of this ellipsoid has a
foliation by $2$-dimensional ellipsoids $E^t:=\partial E(1,b)\cap\{ y_2=t\}$,
$t\in (-b,b)$, outside the two singular points $(0,0,0,\pm b)$,
just as the foliation of $S^3$ by $2$-spheres
$S^t$ that we are going to consider in the proof of the ball theorems
in the next section. Moreover, the relevant energy estimate
in Proposition~\ref{prop:energy-bound} below only depends on
the fact that the projection of $S^t$ to the $x_1y_1$-plane
is contained in the unit disc, which is also true
for the projection of~$E^t$. In other words,
the ball theorems remain true with the convex boundary
component $S^3$ replaced by $\partial E(1,b)$. Now, as before,
this gives the upper bound $\pi$ on the two capacities of~$Z$
and completes the proof of the proposition.
\end{proof}

For a survey on other types of symplectic capacities see~\cite{chls07}.

\begin{rem}
\label{rem:cap}
The same proof applies to show that $c_0(B,\lambda)=\pi$
for any primitive $\lambda$ of $\omega_{\st}$
with $\lambda=\lamst$ near $\partial B^4$. For $\lambda=\lamst$
near $\partial Z$ one can only deduce $c_0(Z,\lambda )\leq\pi$.
\end{rem}

Gromov's celebrated non-squeezing theorem~\cite[p.~310]{grom85}
is now, in dimension~$4$, an immediate consequence of
Proposition~\ref{prop:capacity}.

\begin{cor}[Gromov]
There is a symplectic embedding $B^4_r\hookrightarrow Z_R$
if and only if $r\leq R$.
\end{cor}

\begin{proof}
The $4$-ball $B^4_r$ with the symplectic form $\omega_{\st}$ is
symplectomorphic to the unit ball with the symplectic
form $r^2\omega_{\st}$. Hence $c(B^4_r)=\pi r^2$ by
conformality. Similarly we have $c(Z_R)=\pi R^2$. Now the
result follows from monotonicity.
\end{proof}
\subsection{Quantitative Reeb dynamics}
With the capacities introduced in the preceding section we can
derive some simple quantitative results on shortest Reeb orbits.
Frauenfelder, Ginzburg and Schlenk~\cite[Remark~1.13.3]{fgs05}
show that an upper bound on the period of the shortest closed Reeb orbit on
a compact hypersurface $M\subset (\R^{2n},\omega_{\st})$ of
diameter $\diam(M)$ is $\pi(\diam(M))^2$. We recover their result
in dimension~$4$, where we improve the constant by a
reference to~\cite{jung01}.

\begin{cor}
Let $(M,\alpha)\subset (\R^4,\omega_{\st})$ be
a compact hypersurface of contact type. Then
$\inf(\alpha)\leq(2/5)\pi(\diam(M))^2$.
\end{cor}

\begin{proof}
Since the symplectic form $\omega_{\st}$ is translation-invariant, we
have a contact type embedding of $(M,\alpha)$ into $B^4_r$
for any $r$ greater than the circumradius of~$M$, which
by \cite[p.~257]{jung01} is (in dimension~$4$)
at most equal to $\sqrt{2/5}\,\diam (M)$. Hence
\[ \inf (\alpha)\leq c(B^4_r)=\pi r^2\;\;
\text{for any $r>\sqrt{2/5}\,\diam (M)$.}\qed\]
\renewcommand{\qed}{}
\end{proof}

\begin{rem}
The upper bound $\sqrt{2/5}\,\diam (M)$ for the circumradius
(in dimension~$4$) is optimal; it is attained for the regular $4$-simplex.
\end{rem}

In view of Remark~\ref{rem:cap}, the
same argument with the capacity $c$ replaced by $c_0$ gives
the next corollary.

\begin{cor}
Let $(M,\alpha)\subset (\R^4,\omega_{\st})$ be a compact hypersurface
of restricted contact type. Then $\inf_0(\alpha)\leq
(2/5)\pi(\diam(M))^2$. \hfill\qed
\end{cor}

For star-shaped hypersurfaces we have an alternative estimate.

\begin{cor}
On the star-shaped hypersurface
\[ S:=\{ \sqrt{f(p)}\, p\co p\in S^3\}\subset\R^4,\]
with $f\co S^3\rightarrow\R^+$ a smooth function,
we have $\inf_0(\lamst|_{TS})\leq\pi\max f$.\hfill\qed
\end{cor}

\begin{rem}
As shown in \cite[Section~3.5]{hoze94}, on convex hypersurfaces
in $\R^{2n}$ the minimal period $\inf (\alpha)$ equals the
Hofer--Zehnder capacity. In particular, this provides a lower
bound on $\inf (\alpha)$ in terms of the inradius, i.e.\
the radius of the largest ball that can be embedded in the
domain bounded by the hypersurface. The example of the `Bordeaux bottle'
[loc.\ cit.]\ shows that for the class of hypersurfaces
of restricted contact type there is no lower bound on $\inf (\alpha)$
in terms of the inradius.
\end{rem}
\section{Proof of the ball theorems}
\label{section:proof}
Let $(W,\omega)$ be a symplectic cobordism as in one of the
ball theorems. For the time being, only the conditions (C1) to (C5)
on $(W,\omega)$ common to both theorems will be relevant.
On some collar neighbourhood $[0,\varepsilon)\times M_-\subset W$
of the strongly concave boundary~$M_-$ the symplectic form can
be written as $\omega=\rmd (\rme^s\alpha_-)$. We define a family of
symplectic completions $(\wtW,\omega_{\tau})$
of $(W,\omega)$ along $M_-$ similar to \cite[Section~2.2]{hwz03}
as follows. Consider the family of functions
\[ \TT:=\bigl\{ \tau\in C^{\infty}\bigl( (-\infty,\varepsilon),\R^+\bigr)\co
\tau'>0,\, \tau(s)=\rme^s\;\text{for}\;
s\in [0,\varepsilon)\bigr\} .\]
Let $\wtW$ be the manifold obtained from $W$ by attaching infinite
half-cylinders along the boundary~$M_-$, i.e.
\[ \wtW:=(-\infty,0]\times M_-\cup_{M_-}W,\]
where $M_-\subset\partial W$ is identified with $\{ 0\}\times M_-$ in the
half-cylinder $(-\infty,0]\times M_-$. Then define the symplectic form
$\omega_{\tau}$ on $\wtW$ by
\[ \omega_{\tau}:=\begin{cases}
                  \omega              & \text{on $W$},\\
                  \rmd (\tau\alpha_-) & \text{on $(-\infty,0]\times M_-$}.
                  \end{cases}
\]

Next we choose an almost complex structure $J$ on $\wtW$
compatible with each $\omega_{\tau}$ and subject to the
following conditions:
\begin{itemize}
\item[(J1)] Under the identification of a collar neighbourhood
of $S^3\subset W$ in $(W,\omega)$ with a neighbourhood of
$S^3=\partial D^4$ in $(D^4,\omega_{\st})$, as stipulated by
condition (C5), $J$ looks like the standard complex structure
on $\C^2$.
\item[(J2)] On the cylindrical end $(-\infty,\varepsilon)\times M_-$,
the almost complex structure is cylindrical
and symmetric in the sense of \cite[p.~802, 807]{behwz03}, i.e.\ it
preserves $\xi_-$ and satisfies $J\partial_s=R_{\alpha_-}$.
\item[(J3)] Extend $\xi_+$ to a rank-2 distribution (still
denoted $\xi_+$) in the tangent bundle $TW$ over a neighbourhood of
$M_+$ in~$W$ such that $\omega|_{\xi_+}>0$. Choose $J$ on this
neighbourhood such that $\xi_+$ and its $\omega$-orthogonal
complement are $J$-invariant. In particular, the boundary $M_+$
is then $J$-convex.
\item[(J4)] Outside the regions described in (J1) to (J3),
the almost complex structure is required to be chosen in such a way
that $J$ is regular for spheres (globally on~$\wtW$) in
the sense of \cite[Definition~3.1.4]{mcsa04}, cf.\ the
following remark~(2).
\end{itemize}

\begin{rems}
\label{rems:J}
(1) By \cite[Remark~4.3]{geze10}, condition (J3) implies that $M_+$
can be written as the level set of a smooth function on $W$ that is strictly
plurisubharmonic in a neighbourhood of~$M_+$. Then the maximum
principle holds in that neighbourhood.

(2) A choice of $J$ as required by (J4) is possible by
\cite[Remark~3.2.3]{mcsa04}. By that remark, all that is
required to achieve regularity for spheres is that no sphere
lies entirely in the regions where $J$ is prescribed by one
of the conditions (J1) to (J3). Indeed, no such sphere
can exist, since in all these regions the maximum
principle applies. The proof of the relevant result
\cite[Theorem~3.1.5]{mcsa04} only needs to be modified in
one place in order to account for the non-compactness of~$\wtW$:
instead of requiring the condition $\|\rmd u\|_{\infty}\leq K$
(condition (3.2.3) in~\cite{mcsa04}) to hold globally, we only impose this
condition on curves $u$ with image in $[-K,0]\times M_-\cup_{M_-}W$.
On each of these compact manifolds one has an open and dense set of
regular almost complex structures, and one can then pass to the
intersection of these sets over all~$K>0$ as in~\cite{mcsa04}.

(3) According to \cite[Theorem~3.1.5]{mcsa04}, the dimension
of the moduli space of simple $J$-holomorphic spheres (quotiented by
the $6$-dimensional automorphism group of $S^2=\C P^1$)
in the homology class $A$ is given by $2c_1(A)-2$. Hence, if
$A\in H_2(\wtW;\Z)$ is represented by a non-constant holomorphic sphere,
then $c_1(A)\geq 1$.
\end{rems}

We now want to introduce a moduli space of $J$-holomorphic discs
in $\wtW$ whose boundary is required to lie in $S^3\subset\partial W$
(subject to a varying totally real boundary condition). For this we need to
recall some notation from~\cite{geze10}.

We begin with the unit sphere $S^3$ in $\C^2$ with
complex Cartesian coordinates $(z_1=x_1+iy_1,z_2=x_2+iy_2)$.
Let $H$ be the height function on $S^3$ given by projection onto
the $y_2$-coordinate. For $t\in(-1,1)$ the level sets
$S^t:= H^{-1}(t)$ define a smooth foliation of
$S^3\setminus\{(0,0,0,\pm1)\}$ by $2$-spheres.
We regard the points
\[ q^t_{\pm}:= (0,0,\pm \sqrt{1-t^2},t) \]
as the poles of these $2$-spheres.

This family of poles, together with the two poles
$(0,0,0,\pm 1)$ of $S^3$, forms an unknot
\[ K:= \bigl\{ (0,0,\pm\sqrt{1-t^2},t) \co t\in[-1,1]\bigr\}\]
in~$S^3$. The complement $S^3\setminus K$ is foliated
by circles that bound holomorphic discs
\[ D_s^t:=D^4\cap \bigl(\C\times\{ x_2=s,y_2=t\}\bigr),\;\;
|t|<1,\;\; |s|<\sqrt{1-t^2}.\]
For each $t\in (-1,1)$, the circles $\partial D_s^t$
foliate the punctured $2$-sphere $S^t\setminus\{ q^t_{\pm}\}$.

For $|t|<1$ and $|s|<\sqrt{1-t^2}$ define a smooth
real-valued function
\[ \theta (s,t):= \frac{t}{2\sqrt{1-t^2}}\cdot\ln
\left( \frac{\sqrt{1-t^2}+s}{\sqrt{1-t^2}-s} \right).\]
For each $t$ this defines a diffeomorphism from
$(-\sqrt{1-t^2},\sqrt{1-t^2})$ to~$\R$. Now consider
the parametrisations 
\[ u^t_s(z):= \bigl( \sqrt{1-s^2-t^2}\cdot\rme^{i\theta (s,t)}\cdot z,
                   s,t\bigl),\;\; z\in\D,\]
of the holomorphic discs $D^t_s$.
The rotation factor $\rme^{i\theta (s,t)}$
has been chosen in such a way that for each fixed $t\in (-1,1)$
and $z\in\D$, the map $s\mapsto u^t_s(z)$, $|s|<\sqrt{1-t^2}$,
is a parametrisation of a leaf of the characteristic
foliation on $S^t\setminus\{ q^t_{\pm}\}$ induced by the
standard contact structure $\xist$ on~$S^3$. The three
leaves corresponding to $z=i^k$, $k=0,1,2$, will be denoted by
$\ell^t_k$. These leaves will be used to put a restriction
on three marked points of the holomorphic discs in our moduli space,
which amounts to quotienting out the non-compact $3$-dimensional
automorphism group of~$\D$.

For $|s|$ sufficiently close to $\sqrt{1-t^2}$, the image of the
holomorphic disc $u^t_s$ will lie in the neighbourhood of
$S^3\subset D^4$ that has been identified with a neighbourhood
of $S^3\subset\wtW$. These discs define a relative homotopy
class $A^t\in\pi_2(\wtW,S^t\setminus\{ q^t_{\pm}\})$.
We now always take the holomorphic identification between a neighbourhood of
$S^3=\partial D^4$ in $D^4$ and $S^3\subset\partial\wtW$
in $\wtW$ for granted.

\begin{defn}
A {\bf $t$-level Bishop disc} is a smooth (up to the boundary)
$J$-holo\-mor\-phic map
\[ u^t\co (\D,\partial\D)\lra (\wtW,S^t\setminus\{q^t_{\pm}\}), \]
i.e.\ a solution of the Cauchy--Riemann equation
\[ \partial_xu+J(u)\partial_yu=0,\]
satisfying the following conditions:
\begin{itemize}
\item[(D1)] $[u^t]=A^t\in\pi_2(\wtW,S^t\setminus\{ q^t_{\pm}\})$.
\item[(D2)] $u^t(i^k)\in\ell^t_k$, $k=0,1,2$.
\end{itemize}
The collection
\[ \WW:= \bigl\{ u^t\co t\in (-1,1),\;
u^t\;\text{is a}\; t \text{-level Bishop disc} \bigr\}\]
of all such discs is the {\bf moduli space of Bishop discs}.
\end{defn}

For $\delta\in (0,1)$ we define a neighbourhood of the unknot
$K\subset S^3$ by
\[ \UU^{\delta}:=K\cup\bigl\{ u^t_s(z)\co z\in\partial\D,\, 1-\delta<s^2+t^2
<1\bigr\}\subset S^3.\]
We choose $\delta$ so small that the holomorphic discs
$D^t_s=u^t_s(\D)$ with boundary in $\overline{\UU}^{\delta}$
(i.e.\ with $1-\delta\leq s^2+t^2\leq 1$) lie entirely in
the neighbourhood of $S^3$ in $D^4$ that has been identified
holomorphically with a neighbourhood of $S^3$ in~$\wtW$.
This allows us to regard those $u^t_s$ as holomorphic discs
in~$\wtW$. Then, according to \cite[Corollary~4.9]{geze10}, any $t$-level
Bishop disc whose boundary meets the set $\overline{\UU}^{\delta}$
is one of the standard Bishop discs $u^t_s$. As in our previous paper,
we can therefore introduce the following subset of the moduli space~$\WW$.

\begin{defn}
The {\bf truncated moduli space} is
\begin{eqnarray*}
\lefteqn{\WW^{\delta}:= \bigl\{ u^t \co
t\in [-\sqrt{1-\delta},\sqrt{1-\delta}],\bigr. }\\
& &  \bigl. u^t\;
\text{is a}\; t\text{-level Bishop disc} \;
\text{such that} \; u^t(\partial\D)\subset S^3\setminus
\UU^{\delta} \bigr\} .
\end{eqnarray*}
\end{defn}

The following statements from \cite[Section~4]{geze10}, where such Bishop
discs were studied for $W=D^4$, carry over to the present setting.

\begin{prop}
\label{prop:Bishop}
(a) Every Bishop disc has Maslov index~$2$, i.e.\
$\mu(A^t)=2$ for all $t\in (-1,1)$.

(b) All Bishop discs are embedded and mutually disjoint.\qed
\end{prop}

The same arguments as in \cite{geze10} apply to prove transversality,
i.e.\ that the moduli space $\WW$ is a manifold;
but see the appendix at the end of this paper.

In Section~\ref{section:compact} we shall establish
compactness for the truncated moduli space~$\WW^{\delta}$
under the assumption $\inf(\alpha_-)>\pi$ or, in the exact case,
$\inf_0(\alpha_-)>\pi$. Thus, provided there are no short
Reeb orbits for~$\alpha_-$, the truncated moduli space
is a compact manifold with boundary. Then the proof of
\cite[Proposition~5.1]{geze10} goes through unchanged; this result says
the following.

\begin{prop}
\label{prop:evaluation}
Let $(W,\omega)$ be a symplectic cobordism as in the ball
theorems.
If\/ $\inf(\alpha_-)>\pi$, or  $\inf_0(\alpha_-)>\pi$ in the exact case,
the evaluation map $\ev_1\co u\mapsto u(1)$ defines a
diffeomorphism
\[ \ev_1\co\WW^{\delta}\lra \{ u^t_s(1)\co s^2+t^2\leq 1-\delta\}
=:Q^{\delta}\]
between $\WW^{\delta}$ and the closed $2$-disc $Q^{\delta}\subset S^3$. \qed
\end{prop}

It is now a simple matter to prove Theorems~\ref{thm:ball}
and~\ref{thm:ball-exact}.

\begin{proof}[Proof of the ball theorems]
Under the assumption $\inf(\alpha_-)>\pi$ or $\inf_0(\alpha_-)>\pi$,
respectively, and hence with $\WW^{\delta}$ being established as a closed
disc by the preceding proposition, one can define
an embedding
\[ F\co \bigl(\D\times \Int\D,\partial\D\times \Int\D\bigr) \lra
(\wtW\setminus K,S^3\setminus K)\]
as in \cite[Section~5]{geze10} by setting
\[ F(z,s,t)=
\begin{cases}
\Bigl(\ev_1^{-1}\bigl(u^t_s(1)\bigr)\Bigr)(z) & 
               \quad\text{on}\quad \D\times\{s^2+t^2\leq 1-\delta\},\\[2mm]
u^t_s(z)                                      &
               \quad\text{on}\quad \D\times\{1-\delta\leq s^2+t^2<1\}.
\end{cases}
\]
We also have the standard embedding
\[ F_{\st}\co \bigl(\D\times \Int\D,\partial\D\times \Int\D\bigr) \lra
(D^4\setminus K,S^3\setminus K)\]
given by $F_{\st}(z,s,t)=u^t_s(z)$. Both $F$ and $F_{\st}$ are
holomorphic fillings of $S^3$ in the sense of~\cite[Definition~5.2]{geze10},
and they obviously coincide for $1-\delta\leq s^2+t^2<1$.
It follows that the map $F\circ F_{\st}^{-1}$, {\em a priori\/} defined
on $D^4\setminus K$, equals the identity in a neighbourhood
of $S^3\subset D^4$, and hence extends in the obvious way
to an embedding
\[ (D^4,S^3)\lra (\wtW,S^3),\]
which by the compactness of $D^4$ must be a diffeomorphism.
\end{proof}
\section{The Hofer energy}
\label{section:hofer}
The compactness proof for the truncated moduli space $\WW^{\delta}$
is based on energy estimates for holomorphic discs and spheres
in the almost complex manifold $(\wtW, J)$. The following notion
of energy is essentially the one introduced in \cite[Section~3.2]{hofe93}.

\begin{defn}
Let $\Sigma$ be a Riemann surface (potentially non-compact or with
boundary) and $u\co\Sigma\ra\wtW$ a $J$-holomorphic curve.
The {\bf Hofer energy} of $u$ is
\[ E(u):=\sup_{\tau\in\TT}\int_{\Sigma}u^*\omega_{\tau}.\] 
\end{defn}

The Hofer energy of the holomorphic discs $u^t_s$ in the standard
holomorphic filling of $S^3\subset\C^2$ is uniformly bounded,
cf.~\cite[Section~2.4]{geze10}. We now
prove a sharp estimate for the energy of Bishop discs in~$\wtW$.

\begin{prop}
\label{prop:energy-bound}
The Hofer energy of the Bishop discs in $\wtW$ is uniformly bounded
by~$\pi$, i.e.\ $E(u)\leq\pi$ for all $u\in\WW$.
\end{prop}

\begin{proof}
Let $u=u^t$ be a $t$-level Bishop disc. Choose a
function $\tau\in\TT$. We want to estimate $\int_{\D}u^*\omega_{\tau}$.

By Proposition~\ref{prop:Bishop} the Bishop disc $u$ is an embedding, hence
\[ \int_{\D}u^*\omega_{\tau}=\int_{u(\D)}\omega_{\tau}.\]

The boundary $u(\partial\D)$ of the Bishop disc is
contained in $S^t\setminus\{q^t_{\pm}\}\subset S^3$.
The $2$-sphere $S^t$ is naturally oriented as the unit sphere
in $x_1y_1x_2$-space.
Let $D^t$ be the $2$-disc in $S^t$ (with the induced orientation)
whose oriented boundary equals $u(\partial\D)$;
the disc $D^t$ is characterised by the condition $q^t_+\in D^t$.
The $2$-discs $u(\D)$ and $D^t$ both represent the relative homotopy class
$A^t\in\pi_2(\wtW,S^t\setminus\{q^t_{\pm}\})$, and they
coincide along the boundary. Since $\omega_{\tau}$ is exact near
$S^t$ (and closed on all of~$\wtW$) it follows that
\[ \int_{u(\D)}\omega_{\tau}=\int_{D^t}\omega_{\tau}
=\int_{D^t}\omega_{\st}.\]

On $TS^t$ we have $\omega_{\st}=\rmd x_1\wedge\rmd y_1$. 
So the integral of $\omega_{\st}$ over a subset of $S^t$ measures
the area of the projection of that subset to the $x_1y_1$-plane,
where the regions in the upper hemisphere $\{ x_2\geq 0\}$ are
counted positively; those in the lower hemisphere, negatively.
It follows that $\int_{D^t}\omega_{\st}$, and hence
$\int_{\D}u^*\omega_{\tau}$, is bounded above by
$\pi(1-t^2)$.
\end{proof}
\section{Compactness}
\label{section:compact}
We now want to show that the truncated moduli space $\WW^{\delta}$
is compact. The basic set-up is similar to \cite[Section~6]{geze10}.
We equip $\WW^{\delta}$ with the topology induced by the
$W^{1,p}$-norm, $p>2$, on maps $\D\rightarrow\wtW$;
in \cite[Section~6]{geze10} the range was $D^4\subset\C^2$.

\begin{prop}
\label{prop:compact}
If\/ $\inf (\alpha_-)>\pi$, or $\inf_0(\alpha_-)>\pi$ in the exact case,
the truncated moduli space $\WW^{\delta}$ is compact.
\end{prop}

\begin{proof}
Let $(u_{\nu})$ be a sequence in $\WW^{\delta}$, where
$u_{\nu}$ is of level~$t_{\nu}$.
After passing to a subsequence we may assume that $t_{\nu}\ra t_0\in
[-\sqrt{1-\delta},\sqrt{1-\delta}]$.
As in \cite{geze10} we want to apply \cite[Theorem~B.4.2]{mcsa04} in order
to prove compactness, i.e.\ to find a converging subsequence of $(u_{\nu})$
with respect to the $W^{1,p}$-norm. This requires the following:
\begin{itemize}
\item[(i)] There is a uniform $L^p$-bound for
the sequence $(|\nabla u_{\nu}|)$, where the norm is taken with respect to
some {\em complete\/} metric on~$\wtW$. We claim that the sequence
$(|\nabla u_{\nu}|)$ is uniformly bounded even in the supremum norm
on the closed disc~$\D$. This part of the argument is to some extent
parallel to~\cite{geze10}. In some places we need to invoke additional
energy estimates to compensate for the lack of compactness of the range
of our holomorphic discs. Extra care needs to be taken with
potential bubbling at interior points, because we no longer have
a global maximum principle on~$\wtW$ that would preclude spheres.
\item[(ii)] The image $u_{\nu}(\D)$ stays inside a fixed compact subset of
$\wtW$ for all~$\nu$.
\end{itemize}

First we notice that (ii) is a straightforward consequence of the
bounds we establish in~(i). Indeed, with the help of the
mean value theorem the uniform $C^0$-bound on
the image $u_{\nu}(\D)$ follows from the uniform
bound on the supremum norm of $|\nabla u_{\nu}|$, together
with the fact that $u_{\nu}(\partial\D)$ stays
in the compact subset $S^3\subset\partial\wtW$ of~$\wtW$
(and the assumption that the metric be complete).

Arguing by contradiction, assume that there is no uniform bound
(in $\nu\in\N$) on $\max_{z\in\D}|\nabla u_{\nu}(z)|$.
We can then find a sequence of points
$z_{\nu}\ra z_0$ in $\D$ such that $|\nabla u_{\nu}(z_{\nu})| \ra\infty$.
We distinguish the cases $z_0\in\partial\D$ and $z_0\in\Int\D$.

\vspace{1mm}

{\bf Case 1:} $z_0\in\partial\D$.
Choose a conformal map from $\Hp\cup\{\infty\}$ to $\D$ that sends
$0$ to $z_0$ and $\infty$ to~$-z_0$.
Subject to this conformal identification, we regard the $u_{\nu}$
as maps
\[ u_{\nu}\co (\Hp ,\R )\lra (\wtW,S^{t_{\nu}}\setminus
\{ q^{t_{\nu}}_{\pm}\} ),\]
and the sequence $(z_{\nu})$ as a sequence in $\Hp$ converging to~$0$,
still satisfying
\[ R_{\nu}:=|\nabla u_{\nu} (z_{\nu})| \ra\infty.\]
As shown in the proof of \cite[Proposition~6.1]{geze10}
(by an argument going back to Hofer),
after passing to a subsequence of $(u_{\nu})$ one can find a
sequence $\varepsilon_{\nu}\searrow 0$ such that
\begin{itemize}
\item $\varepsilon_{\nu}R_{\nu}\ra\infty$,
\item $|\nabla u_{\nu}(z)| \leq 2R_{\nu}\;\;\text{for all}\;\;
z\in\Hp \;\;\text{with}\;\; |z-z_{\nu}|\leq\varepsilon_{\nu}$,
\item $R_{\nu}y_{\nu}\ra r\;\;
\text{for some}\;\; r\in [0,\infty]$, where $z_{\nu}=x_{\nu}+iy_{\nu}$.
\end{itemize}

\vspace{1mm}

{\bf Case 1.1:} $r<\infty$.
Here the argument is largely analogous to that in~\cite{geze10}.
One considers the rescaled sequence $(w_{\nu})$ on~$\Hp$, defined by
\[ w_{\nu}(z):= u_{\nu}(x_{\nu}+z/R_{\nu}),\;\; z\in\Hp.\]
The only issue to take care of is the non-compactness
of the range $\wtW$ of the $w_{\nu}$. In fact,
this does not cause any problems,
since there is a uniform bound on the gradient of the $w_{\nu}$,
and the $w_{\nu}$ send $\R=\partial\Hp$ to the compact subset
$S^3\subset\partial\wtW$. So there is a $C^0_{\loc}$-bound
on the $w_{\nu}$, which allows us to apply
\cite[Theorem~B.4.2]{mcsa04} as in~\cite{geze10}.
As there we then find a subsequence of $(w_{\nu})$
(after a modification replacing the varying boundary
condition by a varying almost complex structure)
that converges in $C^{\infty}_{\loc}$ to a non-constant $J$-holomorphic map
\[ w\co (\Hp,\R)\lra (\wtW,S^{t_0}\setminus\{q^{t_0}_{\pm}\}).\]

We now need to show that the singularity of $w$ at $\infty$ can
be removed, i.e.\ that $w$ extends to an honest holomorphic disc
\[ (\D ,\partial\D)\lra (\wtW,
S^{t_0}\setminus\{q^{t_0}_{\pm}\}).\]
Such a disc would have contradictory properties as in~\cite{geze10}.

For this removal of singularities, it is again the non-compactness of $\wtW$
that forces us to take extra care.
Write $|\, .\, |_{\tau}$ for the norm induced by the (incomplete) metric
$g_{\tau}:=\omega_{\tau}(\, .\, ,J\, .\, )$. The {\bf Dirichlet energy}
of $w$ is defined by
\[ \frac{1}{2}\int_{\Hp}|\nabla w|_{\tau}^2\,\mathrm{dvol}_{\Hp}=
\frac{1}{2}\int_{\Hp}\bigl( |\partial_sw|_{\tau}^2+
|\partial_tw|_{\tau}^2\bigr)\,\rmd s\wedge\rmd t,\]
where $z=s+it$. Since $w$ is holomorphic for the $\omega_{\tau}$-compatible
almost complex structure~$J$, the Dirichlet energy of $w$
equals its symplectic energy $\int_{\Hp}w^*\omega_{\tau}$,
see \cite[p.~21]{mcsa04}.
These energies are invariant under conformal
reparametrisations, so Proposition~\ref{prop:energy-bound}
and the $C^{\infty}_{\loc}$-convergence of the sequence $(w_{\nu})$
yield the estimate
\[ \frac{1}{2}\int_{\Hp}|\nabla w|_{\tau}^2\,
\mathrm{dvol}_{\Hp}\leq\pi.\]

Finiteness of the Dirichlet energy is one of the
conditions in the theorem on removal of
singularities~\cite[Theorem~4.1.2]{mcsa04}. In addition, that theorem
requires the image of $w$ to lie in a compact manifold and
$w(\partial\Hp)$ to lie in a Lagrangian submanifold with respect to
a symplectic form taming~$J$.

We first address the latter point. Choose a Riemannian metric $g$ 
on $\wtW$ such that $J$ is orthogonal with respect to $g$,
and such that $J$ maps each tangent space of the totally
real submanifold $L:=S^{t_0}\setminus\overline{\UU}^{\delta/2}$ to its
$g$-orthogonal complement; this is possible by a lemma
of Frauenfelder, see~\cite[Lemma~4.3.3]{mcsa04}.
Define a non-degenerate $2$-form $\sigma$ on $\wtW$ by
$\sigma:=g(J\, .\, ,\, .\, )$. Then the pull-back of $\sigma$
to $L$ vanishes identically, and $J$ is $\sigma$-compatible.

The usual proof of Weinstein's Lagrangian neighbourhood
theorem, see \cite[Theorem~3.33]{mcsa98}, allows one to
find a diffeomorphism $\phi$ from a neighbourhood of $L$ in $\wtW$
to a neighbourhood of the zero section in the cotangent bundle $T^*L$
such that the pull-back $\omega_L:=\phi^*(\rmd\bfp\wedge\rmd\bfq)$ of the
canonical symplectic form on $T^*L$ coincides with $\sigma$ on
$T\wtW|_{L}$. In particular, $J$ is tamed by $\omega_L$
on $T\wtW|_{L}$, and hence in a neighbourhood of $L$.
So $\omega_L$ is the desired symplectic form.

It remains to show that points in $\Hp$ sufficiently close
to $\infty$ are mapped by $w$ into that neighbourhood of~$L$.
By precomposing with the conformal equivalence
\[ \C\supset\R\times [0,\pi]\lra\Hp\setminus\{ 0\},\;\;\;
s+it\mapsto\rme^{s+it},\]
we may regard $w|_{\Hp\setminus\{ 0\}}$ as a $J$-holomorphic
map defined on $\R\times [0,\pi]\subset\C$. In this
parametrisation, neighbourhoods of the singular point are
of the form $\{ s>R\}$.

For $s\in\R$ set
\[ \gamma_s(t):=w(s+it),\;\; t\in[0,\pi].\]
The length $l(s)$ of $\gamma_s$ with respect to the metric $g_{\tau}$ is
\[ l(s)=\int_0^{\pi}|\dot{\gamma}_s|_{\tau}\,\rmd t.\]

Since $w$ maps $\R\times\{0\}$ and $\R\times\{\pi\}$ to
the compact set $S^{t_0}\setminus\UU^{\delta}\subset L$, the fact that $w$
maps a neighbourhood of the singularity to a neighbourhood of $L$
is a consequence of the following lemma.

\begin{lem}
For $\tau=\exp$ we have $\lim_{s\ra\infty}l(s)=0$.
\end{lem}

\begin{proof}
Since $J$ is $\omega_{\tau}$-compatible and $w$ is $J$-holomorphic, we have
\[ |\dot{\gamma}_s|_{\tau}=|\partial_tw|_{\tau}=|\partial_sw|_{\tau},\]
hence $|\dot{\gamma}_s|_{\tau}^2=|\nabla w|_{\tau}^2/2$.
The choice $\tau=\exp$ and condition (J2) give us a curvature bound
on the corresponding metric. This allows us to apply the mean value
inequality \cite[Lemma~4.3.1]{mcsa04}, cf.\ the computations on
page 84 of~\cite{mcsa04}. For $s$ large, the assumptions of that
lemma are satisfied, so there is a constant $C$ depending only on
the geometry of the manifold such that
\begin{eqnarray*}
|\dot{\gamma}_s(t)|_{\tau}^2 & = & \frac{1}{2}|\nabla w(s+it)|_{\tau}^2\\
  & \leq & C\int_{B_1(s+it)\cap(\R\times[0,\pi])} |\nabla w|_{\tau}^2\\
  & \leq & C\int_{[s-1,\infty)\times[0,\pi]} |\nabla w|_{\tau}^2.
\end{eqnarray*}
Hence $|\dot{\gamma}_s(t)|_{\tau}\rightarrow 0$ uniformly in $t$
for $s\rightarrow\infty$.
\end{proof}

\vspace{1mm}

{\bf Case 1.2:} $r=\infty$.
In this case, again as in~\cite{geze10},
we define the rescaled sequence $(w_{\nu})$ by
\[ w_{\nu}(z):= u_{\nu}(z_{\nu}+z/R_{\nu})
\;\;\mbox{\rm for}\;\; z=x+iy\;\; \mbox{\rm with}\;\;
y\geq -y_{\nu}R_{\nu}.\]
Then $|\nabla w_{\nu}(0)|=1$, the Dirichlet energy of the $w_{\nu}$
is bounded by~$\pi$, and we have
the uniform estimate
\[ |\nabla w_{\nu}(z)| \leq 2 \;\; \mbox{\rm for all}\;\;
z\in\C\;\;\mbox{\rm with}\;\;
|z|\leq\varepsilon_{\nu}R_{\nu}\;\;\mbox{\rm and}\;\;
y\geq -y_{\nu}R_{\nu}.\]

\vspace{1mm}

{\bf Case 1.2.a:} The sequence $(w_{\nu}(0)=u_{\nu}(z_{\nu}))$ converges
(after passing to a subsequence).

Because of the uniform bound $|\nabla w_{\nu}|\leq 2$ on
the exhausting sequence
\[ K_{\nu}:=\{ z\in\C\co |z|\leq\varepsilon_{\nu}R_{\nu},\;
y\geq -y_{\nu}R_{\nu}\} \]
of compact subsets of~$\C$, the convergence of $(w_{\nu}(0))$
implies that we have a $C^0$-bound on~$w_{\nu}$ on the
compact set $K_{\nu}$, which allows us again to
apply~\cite[Theorem~B.4.2]{mcsa04}.
This now gives us a subsequence of $(w_{\nu})$
that converges in $C^{\infty}_{\loc}$ to
a non-constant holomorphic map $w\co\C\rightarrow\wtW$
with $E(w)\leq\pi$. So $w$ is a finite energy plane in the
sense of~\cite{hwz03}.

By \cite[Proposition~2.11]{hwz03} we now have the following
alternative:
\begin{itemize}
\item[(A1)] A sphere bubbles off: the image of $w$ is bounded, and
$w$ has a smooth extension over $\infty$ to a holomorphic sphere.
\item[(A2)] A plane bubbles off: the image of $w$ is unbounded, and
there exists an $r_0>0$ such that
\[ w(z)=:(a(z),f(z))\in (-\infty ,0]\times M_-\;\;\text{for}\;\; |z|\geq r_0.\]
\end{itemize}
Moreover, in the case of alternative (A2), there exists a
sequence $r_{\mu}\ra\infty$, $r_{\mu}\geq r_0$, and a negative number $T<0$
such that
\[ a(r_{\mu}\rme^{2\pi it})\ra -\infty\;\;\;\text{and}\;\;\;
\gamma_{\mu}(t):=f(r_{\mu}\rme^{2\pi it})\ra\gamma (Tt)\]
in $C^{\infty}(\R/\Z,(-\infty,0])$ and $C^{\infty}(\R/\Z,M_-)$,
respectively, for some $|T|$-periodic Reeb orbit $\gamma$ of~$\alpha_-$.

The next lemma is essentially contained in \cite[Theorem~31]{hofe93},
but we give a more direct proof in the present context.

\begin{lem}
\label{lem:period}
The period $|T|$ in alternative\/ {\rm (A2)} satisfies $|T|\leq\pi$.
\end{lem}

\begin{proof}
For each $\mu\in\N$ we choose a compactly supported function
\[ \tau_{\mu}\in C^{\infty}\bigl( (-\infty,\varepsilon),\R_0^+)\]
with the following properties:
\begin{itemize}
\item[(i)] $\tau_{\mu}'\geq 0$,
\item[(ii)] $\tau_{\mu}(s)=\rme^s$ for $s\in[0,\varepsilon)$,
\item[(iii)] $\tau_{\mu}=1-1/\mu$ on $a(\{ |z|=r_{\mu}\})$.
\end{itemize}
Notice that $\tau_{\mu}$ is an element of the $C^{\infty}$-closure of the
set $\TT$ of functions used to define the Hofer energy.

We then compute
\begin{eqnarray*}
E(w) & \geq & \int_{|z|\geq r_{\mu}} w^*\omega_{\tau_{\mu}}
     \;\;\; = \;\;\; \int_{|z|\geq r_{\mu}} w^*(\rmd (\tau_{\mu}\alpha_-))\\
     & =    & -\bigl(1-\frac{1}{\mu}\bigr)\,\int_{|z|=r_{\mu}} f^*\alpha_-\\
     & =    & -\bigl(1-\frac{1}{\mu}\bigr)\,\int_{\gamma_{\mu}}\alpha_-
     \;\;\; = \;\;\; -\bigl(1-\frac{1}{\mu}\bigr)\,
              \int_0^1 \alpha_-(\dot{\gamma}_{\mu}(t))\,\rmd t\\
     & \stackrel{\mu\ra\infty}{\lra} &
              -\int_0^1 \alpha_-(T\dot{\gamma}(Tt))\,\rmd t\\
     & =    & -T\;\;\; = \;\;\; |T|,
\end{eqnarray*}
where we have used the theorem of Stokes in the second line.
Since $E(w)\leq\pi$, this proves the lemma.
\end{proof}

This lemma shows that our assumption $\inf (\alpha_-)>\pi$
in Proposition~\ref{prop:compact} precludes alternative~(A2). 
In the exact case, where we only require $\inf_0 (\alpha_-)>\pi$,
we rule out (A2) as follows.
Define the collar neighbourhood $[0,\varepsilon )\times M_-\subset W$
of $M_-$ by the flow of the Liouville vector field $Y$ given by
$i_Y\rmd\lambda =\lambda$. Then $\lambda$ on $W$ and $\tau_{\mu}\alpha_-$
on $(-\infty ,0]\times M_-$ glue to a global primitive $\lambda_{\mu}$
of $\omega_{\tau_{\mu}}$ on~$\wtW$. Under alternative (A2),
the $2$-form $w^*\omega_{\tau_{\mu}}=w^*(\rmd\lambda_{\mu})$
on $\C$ would be compactly supported, hence
\[ \int_{|z|\geq\tau_{\mu}} w^*\omega_{\tau_{\mu}}\leq
\int_{\C}w^*(\rmd\lambda_{\mu}) =0\]
by the theorem of Stokes, which would imply $T=0$
by the computation in the preceding lemma.

In the exact case, alternative (A1) is likewise impossible,
since by Stokes there are no non-constant holomorphic spheres
in an exact symplectic manifold.

This concludes the discussion of Case 1.2.a, except for the
potential bubbling of spheres in the non-exact case.

\vspace{1mm}

{\bf Case 1.2.b:} The sequence $(w_{\nu}(0))$ is of the form
$w_{\nu}(0)=(a_{\nu}(0),f_{\nu}(0))\in (-\infty,0]\times M_-$
(in the notation of alternative~(A2)) with $a_{\nu}(0)\ra
-\infty$ (again possibly after passing to a subsequence).

In this case we use a trick from~\cite{hwz03} to produce a finite
energy plane in the symplectisation $(\R\times M_-,\rmd (\rme^s\alpha_-))$
of~$M_-$. Let $R_{\nu}'$ be the maximal radius
$\leq \varepsilon_{\nu}R_{\nu}$ such that with
$K_{\nu}$ as in Case~1.2.a and
\[ K_{\nu}':= K_{\nu}\cap\{ |z|\leq R_{\nu}'\}\]
we have
\[ w_{\nu}(K_{\nu}')\subset (-\infty,0]\times M_-.\]
Because of $\varepsilon_{\nu}R_{\nu}\ra
\infty$, the uniform estimate $|\nabla w_{\nu}|\leq 2$ on $K_{\nu}$,
and $a_{\nu}(0)\ra-\infty$, the mean value theorem implies
$R_{\nu}'\ra\infty$.

Now consider the shifted sequence
\[ (a_{\nu}-a_{\nu}(0),f_{\nu})\in C^{\infty}(K_{\nu}',\R\times M_-). \]
We continue to write $(w_{\nu})$ for this sequence. By the
compactness of $M_-$ the sequence $(w_{\nu}(0))$ in $\{0\}\times M_-$
has a convergent subsequence, so just as in Case~1.2.a we may
now apply \cite[Theorem~B.4.2]{mcsa04} to obtain
a finite energy plane $w\co\C\ra\R\times M_-$. Specifically,
with the Hofer energy now defined as
\[ E(w)=\sup_{\tau}\int w^*\rmd (\tau\alpha_-),\]
where the supremum is taken over the set
\[ \{\tau\in C^{\infty}(\R,[0,1])\co \tau'\geq 0\} ,\]
we have $E(w)\leq\pi$.

This places us, once again, in the setting of \cite[Proposition~2.11]{hwz03}.
In the symplectisation $(\R\times M_-,\rmd (\rme^s\alpha_-))$
the maximum principle holds, so alternatives (A1) and (A2)
are excluded; instead we must have the following, where $w=(a,f)$:
\begin{itemize}
\item[(A3)] Breaking: there exists a sequence $r_{\mu}\ra\infty$
and a positive number $T>0$ such that
\[ a(r_{\mu}\rme^{2\pi it})\ra\infty \;\;\;\text{and}\;\;\;
\gamma_{\mu}(t):=f(r_{\mu}\rme^{2\pi it})\ra\gamma (Tt)\]
in $C^{\infty}(\R/\Z,\R)$ and $C^{\infty}(\R/\Z,M_-)$,
respectively, for some $T$-periodic Reeb orbit $\gamma$ of~$\alpha_-$.
\end{itemize}

Notice that this $\gamma$ is now a {\em contractible\/}
periodic orbit, since the whole energy plane can be
projected into~$M_-$. Next we estimate the symplectic energy
by choosing $\tau$ to be identically equal to~$1$:
\[ E(w)\geq\int_{|z|\leq r_{\mu}} w^*(\rmd\alpha_-)
=\int_{|z|=r_{\mu}} f^*\alpha_-\stackrel{\mu\ra\infty}{\lra} T.\]

This implies that the assumption $\inf_0(\alpha_-)>\pi$
suffices to rule out alternative~(A3).

\vspace{1mm}

{\bf Case 2:} $z_0\in\Int\D$. This case is completely analogous
to Case 1.2.

\vspace{2mm}

Thus, at this point we have completed the proof of
Proposition~\ref{prop:compact} in the exact case,
and hence the proof of Theorem~\ref{thm:ball-exact}. In the
non-exact case it still remains to show that
no spheres can bubble off, i.e.\ that
alternative (A1) in Case~1.2.a or Case~2 never happens.

Since the maximum principle applies near the convex boundary
of~$\wtW$, all potential bubbling spheres have to be
disjoint from a neighbourhood of that boundary.
Moreover, our arguments have shown that no breaking
or bubbling-off of planes can occur, in particular
near the boundary. Therefore the compactness
result \cite[Theorem~10.2]{behwz03} for an almost complex
manifold without boundary and with cylindrical ends
applies to our situation (modulo a remark that
we shall make presently). That compactness result says that the
sequence $(u_{\nu})$ has a subsequence convergent to a holomorphic building
of height $k_-|1$ with $k_-\geq 0$. Any holomorphic building
coming from a disc and having height $k_-|1$ with $k_->0$ would contain at
least one finite energy plane, whose existence we have excluded.
So the limit is a holomorphic building
of total height~$1$. In other words, the subsequence is Gromov-convergent
to a stable $J$-holomorphic
map $\{ u^j\}_{j=0,\ldots,n}$ in the sense of
\cite[Definition~5.1.1]{mcsa04}. Here our labelling is chosen
such that $u^0$ is a $J$-holomorphic disc, and $u^1,\ldots, u^n$
are $J$-holomorphic spheres. The disc and spheres form
a bubble tree; in particular each sphere has at least one
point of intersection with some other sphere or the disc.

\begin{rem}
The compactness theorem from \cite{behwz03} only applies in the
case that $\alpha_-$ is non-degenerate (i.e.\ the linearised Poincar\'e
return map along closed Reeb orbits of $\alpha_-$, including
multiples, does not have an eigenvalue~$1$).

If $\alpha_-$ is a degenerate contact form with
$\inf (\alpha_-)>\pi$, we argue as in the final paragraph
of~\cite{ach05}. Choose a sequence of smooth functions
$f^{(\mu )}\co M_-\rightarrow\R^+$ converging in $C^{\infty}$ to
the constant function~$1$ and with $f^{(\mu )}\alpha_-$ non-degenerate
for all $\mu\in\N$. If $\inf (f^{(\mu )}\alpha_-)\leq\pi$ for all~$\mu$,
then the argument in \cite{ach05} would show that, likewise,
$\inf (\alpha_-)\leq\pi$, contradicting our assumption.
So we find a function $f$ arbitrarily $C^{\infty}$-close to
$1$ with $f\alpha$ non-degenerate and $\inf (f\alpha_-)>\pi$. This contact
form can be realised on the boundary $M_-$ of the cobordism $W$
by a small modification of $W$ in a collar neighbourhood
of $M_-\subset W$ (after adding a small piece
$(-\varepsilon ,0]\times M_-$ of the symplectisation
to~$W$). Then the whole argument (including the part that
follows below) can be applied to this modified~$W$, showing
$W$ to be a $4$-ball.
\end{rem}

Returning to our purported bubble tree,
we now compute with intersection numbers as in~\cite{ye98}
in order to show that such a bubble tree (with $n\geq 1$)
cannot exist. Arguing by contradiction, we assume
$n\geq 1$, i.e.\ that at least one sphere bubbles off.
Each disc $u_{\nu}$ represents a relative
homotopy class $A^{t_{\nu}}\in \pi_2(\wtW,
S^{t_{\nu}}\setminus\{ q^{t_{\nu}}_{\pm}\})$
of self-intersection number (as defined in \cite[Section~8]{geze10})
$A^{t_{\nu}}\bullet A^{t_{\nu}}=0$; see the proof of
\cite[Proposition~4.5]{geze10}.  Gromov convergence implies that, for
$\nu$ large enough, the homotopy class
of $u_{\nu}$ in $\pi_2(\wtW, S^3\setminus K)$ equals that
represented by the limiting bubble tree. For the purpose of computing
intersection numbers we may assume that the bubble tree and the
$u_{\nu}$ represent the same class $A\in\pi_2(\wtW,
S^{t_0}\setminus\{ q^{t_0}_{\pm}\})$ for some~$t_0$ (since for discs
in different levels we are back to classical intersection theory at interior
points).

By positivity of intersections \cite[Theorem~9.2]{geze10}
and $u_{\nu}\bullet u_{\nu}=0$ we have
\[ u^j\bullet u_{\nu}\geq 0\;\;\text{for}\;\; j=0,\ldots ,n.\]
Since the intersection product is a homotopy invariant, we have
\[ u^j\bullet A\geq 0\;\;\text{for}\;\; j=0,\ldots ,n.\]
From
\[ 0=A\bullet A=\sum_{j=0}^nu^j\bullet A\geq 0\]
we conclude $u^j\bullet A=0$, $j=0,\ldots ,n$.

Again by positivity of intersections we have $u^1\bullet u^k\geq 0$
for $k=0,2,\ldots ,n$, and at least one of these intersection
numbers (corresponding to a neighbour of $u^1$ in the bubble
tree) is positive. Hence $u^1\bullet u^1<0$.
Beware that the intersection number in~\cite{geze10} is weighted differently
from the standard intersection number of closed submanifolds, but the
inequality $u^1\bullet u^1<0$ remains true if $\bullet$
is now interpreted in that standard way.

Without loss of generality we may assume that $u^1$ is
simple, i.e.\ not multiply covered (otherwise apply the following
argument to the corresponding simple sphere).
Then the adjunction inequality~\cite[Theorem~2.6.4]{mcsa04} says that
\[ u^1\bullet u^1-c_1(u^1)+2\geq 0,\]
with equality if and only if $u^1$ is embedded.
By Remark~\ref{rems:J}~(3) we have $c_1(u^1)\geq 1$, and hence
$u^1\bullet u^1\geq -1$. 

We conclude that $u^1\bullet u^1= -1$. Then further
$c_1(u^1)=1$, and equality holds in the adjunction formula.
This means that $u^1$ is an exceptional sphere
in $(\wtW,\omega_{\tau})$ for any choice of
$\tau\in\TT$. If we take $\tau (s)=\rme^s$ on $(-\infty ,\varepsilon)$,
for instance, the vector field $\partial_s$ is a Liouville
vector field for $\omega_{\tau}$ on $(-\infty ,\varepsilon)\times M_-$,
and its flow can be used to push $u^1$ into $(W,\omega)$.
Exceptional spheres in $(W,\omega)$, however, are
excluded by assumption~(C1).

This finishes the proof of Proposition~\ref{prop:compact}
and hence that of Theorem~\ref{thm:ball}.
\end{proof}
\section{Symplectic fillings of $S^2\times S^1$}
\label{section:S2S1}
An obvious strong symplectic filling of
$S^2\times S^1\subset \R^3\times S^1$ with its
standard contact structure $\xist=\ker\bigl(\lamst|_{T(S^2\times S^1)}\bigr)$,
where
\[ \lamst:=\frac{1}{2} (x\,\rmd y-y\,\rmd x)+z\,\rmd\theta,\]
is given by
$(D^3\times S^1,\rmd x\wedge\rmd y+\rmd z\wedge\rmd\theta)$.
The following result is implicit in~\cite{elia90}; for
the uniqueness of the filling up to symplectic deformation
equivalence see~\cite{niwe11}.
For more on the topology
of symplectic and Stein fillings see~\cite[Chapter~12]{ozst04}.

\begin{thm}
Any minimal weak symplectic filling of $(S^2\times S^1,\xist)$
is diffeomorphic to $D^3\times S^1$.
\end{thm}

\begin{proof}
We only present the main outline of the argument;
the details are then completely analogous to those used
for the holomorphic filling of $(\wtW,J)$ in Section~\ref{section:proof}.

The contact manifold $(S^2\times S^1,\xist)$ is foliated by
the $2$-spheres $S^{\theta}:=S^2\times\{\theta\}$, whose characteristic
foliation looks like that of the level spheres $S^t$ in
Section~\ref{section:proof}. The singular points $q^{\theta}_{\pm}$
of these characteristic foliations form two circles $\{z=\pm 1\}\times S^1$.

Suppose $(W,\omega )$ is a weak symplectic filling
of $(S^2\times S^1,\xist)$. Choose an almost complex
structure $J$ on $W$ that satisfies conditions (J3) and~(J4).
In order to define $\theta$-level Bishop discs,
we formulate a condition analogous to (D2) by
choosing three leaves of the characteristic foliation
of the $S^{\theta}$ (in an $S^1$-invariant family, say).
As regards the homotopical condition~(D1), {\em a priori\/}
we may have to consider two families of relative homotopy
classes $A^{\theta}_{\pm}$ defined by standard Bishop discs
in $(W,J)$ near the singular points~$q^{\theta}_{\pm}$.
Because of $A^{\theta}_+\bullet A^{\theta}_-=0$ and positivity
of intersections, however, it follows that any two
$\theta$-level Bishop discs in these two families are
either disjoint or they coincide. This implies that it suffices
to formulate (D1) only in terms of~$A^{\theta}_+$, say.

The corresponding truncated moduli space can then be shown to be
diffeomorphic to $S^1\times D^1$ via the evaluation map $\ev_1$ just
as in Proposition~\ref{prop:evaluation}, and the proof then concludes
like that of the ball theorems. The compactness argument remains
unchanged; a suitable replacement for the bound $\pi$ in the
energy estimate in Proposition~\ref{prop:energy-bound}
is provided by ${\displaystyle\max_{\theta}}
\int_{S^2\times\{\theta\}}|\omega|$.
\end{proof}

We close with a computation of the capacities $c,c_0$
in this context.

\begin{prop}
Let $(V,\omega)$ be a minimal strong symplectic
filling of $(S^2\times S^1,\alpha_{\st})$,
i.e.\ there is a Liouville vector field $Y$ defined near
and pointing outwards along $\partial V=S^2\times S^1$
such that $(i_Y\omega)|_{T(S^2\times S^1)}=\alpha_{\st}$.
Then $c(V,\omega)=\pi$.
If, moreover, $\omega=\rmd\lambda$ with
$\lambda|_{T(S^2\times S^1)}=\alpha_{\st}$, then
$c_0(V,\lambda)=\pi$.
\end{prop}

\begin{proof}
By assumption the symplectic form $\omega$ is the standard
one near $S^2\times S^1$, so in this case we
actually have the bound $\pi$ in the energy estimate
in Proposition~\ref{prop:energy-bound}.
The filling result above implies that any closed hypersurface $M$
in $\Int V$ is separating. The
analogues of the ball theorems for the resulting cobordism
from $M$ (which has to be a concave end)
to $S^2\times S^1$ show that the capacities $c,c_0$ are at
most equal to~$\pi$. On the other hand, the Reeb vector field
of $\alpha_{\st}$ is given by
\[ R_{\alpha_{\st}}=\frac{2}{1+z^2}(x\,\partial_y-y\,\partial_x+z\,
\partial_{\theta}),\]
whose minimal period is equal to~$\pi$, corresponding to the
contractible orbits
\[ \gamma (t)=(\cos 2t,\sin 2t,0,\theta_0)\in S^2\times S^1\subset
\R^3\times S^1,\; t\in [0,\pi].\]
For $r<1$ sufficiently close to $1$ we have a (strict)
contact type embedding of $S^2_r\times S^1$ (with the contact form
induced by $\lamst$) into $(V,\omega)$. So the lower bound $\pi$ on the
capacities in question follows from an exhaustion argument as
in the proof of Proposition~\ref{prop:capacity}.
\end{proof}

\begin{ack}
We thank Peter Albers, Max D\"orner, Felix Schlenk and Chris Wendl
for comments on an earlier version of this paper, and Bernd Kawohl for
reference~\cite{jung01}.
\end{ack}

{\footnotesize
\section*{Appendix}
\label{section:appendix}
Here are some corrections to our
previous paper~\cite{geze10}.

(1) In the first case, step~(ii), of the proof of Proposition~6.1,
$u$ should be replaced by $w$ in three instances.

(2) In the second case of the proof of Proposition~6.1, the uniform
estimate $|\nabla w_{\nu}(z)|\leq 2$ holds
on the set $K_{\nu}$ defined in Case 1.2.a of the present paper.

(3) Proposition~7.3 states that the space $\BB$ of level-preserving
discs $(\D,\partial\D )\ra (\R^4, S^t\setminus\{q^t_{\pm}\})$,
$t\in (-1,1)$, with the homotopical boundary condition~(D1)
from Section~\ref{section:proof} above,
is a Banach manifold. This statement is correct; the proof in \cite{geze10}
shows that it is a Banach manifold modelled on the Banach space
of $W^{1,p}$-sections $\eta$ of $u^*(T\R^4,TS^3)$, where
$u$ is a $W^{1,p}$-map $(\D,\partial\D )\ra (\R^4,S^3\setminus K)$,
that satisfy the additional requirement
\[ \langle\nabla H\circ u|_{\partial\D},\eta|_{\partial\D}\rangle\equiv
\text{const.},\]
with $H$ the height function on $S^3$
as in Section~\ref{section:proof}.

In the proof of that Proposition~7.3 in \cite{geze10} we tried to show more,
namely, that $\BB$ is a Banach submanifold of the Banach manifold
$\CC$ of all $W^{1,p}$-maps $(\D,\partial\D)\ra(\R^4,S^3\setminus K)$
satisfying the corresponding homotopical boundary condition.
This would require the subspace $T_u\BB\subset T_u\CC$ to split.
Since solutions of the boundary value problem (P) in \cite{geze10}
need not be of class $W^{1,p}$, our argument does not prove the existence of
a splitting. This stronger statement, however, is never used in \cite{geze10}
or the present paper.
}

\end{document}